\documentclass[reqno]{amsart}
\usepackage[utf8]{inputenc}
\usepackage[T1]{fontenc}
\usepackage{fourier}
\usepackage{amssymb , amsmath, amsthm}

\usepackage{graphicx}
\usepackage{mathtools}
\mathtoolsset{showonlyrefs}

\newtheorem{defi}{Definition}
\newtheorem{theorem}[defi]{Theorem}

\newtheorem{remark}[defi]{Remark}
 \newtheorem{prop}[defi]{Proposition}
\newtheorem{lemma}[defi]{Lemma}
\newtheorem{cor}[defi]{Corollary}

\newtheorem*{conjecture}{Conjecture}

\usepackage{color}

\newcommand\blue[1]{\textcolor{blue}{#1}}

\usepackage{marginnote}
\let\oldmarginnote\marginnote
\renewcommand{\marginnote}[1]{\oldmarginnote{\blue{\tiny #1}}}

\newcommand{\D}{\mathbb{D}}
\newcommand{\R}{\mathbb R}
\newcommand{\N}{\mathbb N}
\newcommand{\Z}{\mathbb Z}

\newcommand{\T}{\mathbb{T}}

\newcommand{\de}{\, \mathrm{d}}
\newcommand{\del}{\partial}

\newcommand{\CC}{\mathcal{C}}
\newcommand{\CD}{\mathcal{D}}

\newcommand{\CH}{\mathcal{H}}

\newcommand {\bx}{\mathbf x}

\newcommand {\bk}{\mathbf k}

\newcommand {\by}{\mathbf y}

\newcommand {\bn}{\mathbf n}

\newcommand {\RH}{\mathrm H}
\newcommand {\RL}{\mathrm L}
\newcommand {\RB}{\mathrm B}
\newcommand {\RV}{\mathrm V}

\newcommand{\btheta}{\boldsymbol{\theta}}
\newcommand{\abs}[1]{\left\lvert #1 \right\rvert}
\newcommand{\set}[1]{\left\{ #1 \right\}}

\newcommand{\norm}[1]{\left\| #1 \right\|}
\newcommand{\bo}\boldsymbol{}
\newcommand{\bigo}[2][]{O_{#1}\left( #2 \right)}
\newcommand{\smallo}[2][]{o_{#1}\left( #2 \right)}

\DeclareMathOperator{\dist}{dist}

\renewcommand{\hat}{\widehat}
\renewcommand{\tilde}{\widetilde}
\newcommand{\eps}{\varepsilon}
\renewcommand{\phi}{\varphi}

\title[From Steklov to Neumann via homogenisation]{From Steklov to Neumann via homogenisation}
\author{Alexandre Girouard}
\address{D\'epartement de math\'ematiques et de statistique, Pavillon Alexeandre-Vachon, Universit\'e Laval, Qu\'ebec, QC, G1V 0A6, Canada}
\email{alexandre.girouard@mat.ulaval.ca}

\author{Antoine Henrot}
\address{Universit\'e de Lorraine, CNRS, IECL, F-54000 Nancy, France}
\email{antoine.henrot@univ-lorraine.fr}

\author{Jean Lagacé}
\address{Department of Mathematics, University College London, Gower Street, London, WC1E 6BT, United Kingdom}
\email{j.lagace@ucl.ac.uk}

\begin{document}
\begin{abstract}
We study a new link between the Steklov and Neumann eigenvalues of domains in
Euclidean space. This
is obtained through an homogenisation limit of the Steklov problem on a
periodically perforated domain, converging to a family of eigenvalue problems
with dynamical boundary conditions. For this problem, the spectral parameter
appears both in the interior of the domain and on its boundary. This
intermediary problem interpolates between Steklov and Neumann eigenvalues of the
domain. As a corollary, we recover some isoperimetric type bounds for Neumann
eigenvalues from known isoperimetric bounds for Steklov eigenvalues. The
interpolation also leads to the construction of planar domains with first
perimeter-normalized Stekov eigenvalue that is larger than any previously known
example.  
  The proofs are based on a modification of the energy method. It requires quantitative estimates for norms of harmonic functions. An intermediate step in the proof provides a homogenisation result for a transmission problem.
\end{abstract}

\maketitle

%\tableofcontents

\section{\bf Introduction}

Let $\Omega\subset\R^d$ be a bounded and connected domain with
Lipschitz boundary 
$\partial\Omega$. Consider on $\Omega$ the Neumann eigenvalue problem
\begin{gather}\label{Problem:Neumann}
\begin{cases}
-\Delta f= \mu f&\mbox{ in }\Omega,\\
\partial_\nu f= 0 &\mbox{ on }\partial\Omega,
\end{cases}
\end{gather}
as well as the Steklov eigenvalue problem
\begin{gather}\label{Problem:Steklov}
\begin{cases}
\Delta u=0&\mbox{ in }\Omega,\\
\partial_\nu u=\sigma u&\mbox{ on }\partial\Omega.
\end{cases}
\end{gather}
Here $\Delta$ is the Laplacian, and $\del_\nu$ is the outward pointing normal derivative. Both problems consist in finding the eigenvalues $\mu$ and $\sigma$ such that there exist non-trivial smooth solutions to the boundary value problems \eqref{Problem:Neumann} and \eqref{Problem:Steklov}. For both problems, the spectra form discrete unbounded sequences
$$ 0 = \mu_0 < \mu_1 \le \mu_2 \le \dotso \nearrow \infty$$
and
$$0=\sigma_0<\sigma_1\leq\sigma_2\leq\cdots\nearrow\infty,$$
where each eigenvalue is repeated according to multiplicity. The corresponding
eigenfunctions $\set{f_k}$ and $\set{u_k}$ have
natural normalisations as orthonormal bases of $\RL^2(\Omega)$ and $\RL^2(\del\Omega)$, respectively.

\subsection{From Steklov to Neumann : heuristics}
Let us start by painting with broad brushes the relationships between the Neumann
and Steklov eigenvalue problems. They exhibit many similar features, and it is not a surprise that they do so. Indeed, in both cases the eigenvalues are those of a differential or pseudo-differential operator, namely the Laplacian and the Dirichlet-to-Neumann map, whose kernels consist of constant functions. Moreover, in both cases, the natural isoperimetric type problem consists in maximizing $\mu_k$ and $\sigma_k$ (instead of minimizing it as is usual for the Dirichlet problem). The relation
between the two boundary value problem is not solely heuristic and incidental.
Indeed, it is known from the works of Arrieta--Jim\'enez-Casas--Rodriguez-Bernal \cite{Arrieta2008} and Lamberti--Provenzano \cite{LambertiProvenzano2015,LambertiProvenzano2017} that
one can recover the Steklov problem as a limit of weighted Neumann
problems
\begin{gather}\label{Problem:Neumanndensity}
\begin{cases}
-\Delta f= \mu \rho_\eps f&\mbox{ in }\Omega,\\
\partial_\nu f= 0 &\mbox{ on }\partial\Omega,
\end{cases}
\end{gather}
where $\rho_\eps$ is a density function whose support converges to the boundary
as $\eps \to 0$. If we are to interpret the Neumann problem as finding the
frequencies and modes of vibrations of a free boundary membrane, this means that
the Steklov problem represents the frequencies and modes of a membrane whose
mass is concentrated at the boundary. The reader should also refer to the
work of Hassannezhad--Miclo~\cite[Section 4]{HassannezhadMiclo}, where a
similar construction is used to prove a Cheeger-type inequality for Steklov
eigenvalues of a compact Riemannian manifold with boundary.

Our primary goal in this paper is to establish a link in the reverse
direction, by realizing the Neumann problem as a limit of appropriate Steklov
problems. This is achieved in two steps. The first one is to accumulate uniformly
distributed boundary elements inside the domain 
$\Omega$. {This is done by perforating the interior of the domain with small holes that are uniformly distributed}. On these new boundary components, we consider the Steklov boundary
conditions; this step is known as
\emph{the homogenisation process}. We assume that the ratio of
the radii of the holes to the distance between them is at a Cioranescu--Murat type  critical regime. Then, the eigenvalues and eigenfunctions of the Steklov problem converge to those of a \emph{dynamical eigenvalue problem}
  \begin{gather}\label{problem:homo}
    \begin{cases}
      -\Delta U=  A_{d} \beta  \Sigma U&\mbox{ in }\Omega,\\
      \partial_\nu U=\Sigma U&\mbox{ on }\partial\Omega,
    \end{cases}
  \end{gather}
  where $\beta \ge 0$ is the critical regime parameter and $A_d$ is the area of the unit sphere
  in $\R^{d}$. Its eigenvalues form a discrete unbounded sequence: $$\Sigma_{0,\beta}<\Sigma_{1,\beta}\leq\Sigma_{2,\beta}\leq\cdots\nearrow\infty,$$
and once again the functions associated to the eigenvalue $\Sigma_{0,\beta} = 0$ are constant. 

\begin{remark}
In \cite{BelowFrancois}, Joachim von Below and Gilles François studied an eigenvalue
problem that is equivalent
  to Problem~\eqref{problem:homo}, which stems from a parabolic equation
  with dynamical boundary conditions. Indeed, they study the eigenvalue
  problem
  \begin{equation}
  \label{eq:vbf}
   \begin{cases}
   - \Delta u = \lambda u & \text{in }  \Omega,\\
    \del_\nu u = \lambda \alpha u & \text{on } \del \Omega.
   \end{cases}
  \end{equation}
If $\lambda$ is an eigenvalue of Problem~\eqref{eq:vbf}, then $\Sigma=\alpha^{-1} \lambda$ is an eigenvalue of Problem~\eqref{problem:homo} with parameter $\beta = \frac{1}{\alpha A_d}$.
\end{remark}

The parameter $\beta$ in \eqref{problem:homo} can be interpreted as a weight on the interior of the domain,
with the boundary $\del \Omega$ having constant weight $1$. In order to recover
the Neumann problem, the second step will therefore be to send the parameter 
$\beta$ to $\infty$, putting all the weight inside the domain. Under an
appropriate normalisation, eigenvalues and eigenfunctions of Problem
\eqref{problem:homo} converge to those of Problem \eqref{Problem:Neumann},
completing the circle for the relation between the Steklov and the Neumann
problems.
\subsection{The homogenisation process}
Consider a family of problems obtained by
removing periodically placed balls from the domain $\Omega$.
More precisely, given $0 < \eps < 1$, and $\bk\in\Z^d$, define the cube
$$Q_{\bk}^\eps := \eps \bk+\left[-\frac{\eps}{2},\frac{\eps}{2}\right]^d\subset\R^d,$$
and define the set of indices
\begin{gather*}
I^\eps := \set{\bk\in\Z^d\,:\,Q_{\bk}^\eps \subset\Omega}.
\end{gather*}
Let $r_\eps$ be  an increasing positive function of $\eps$ with $r_\eps < \eps/2$.
For $\bk\in \Z^d$, define 
$$T_{\bk}^\eps:=B\left(\eps\bk,r_\eps\right)\subset Q_{\bk}^\eps$$
and set
\begin{equation}
 T^\eps := \bigcup_{\bk \in I^\eps} T^\eps_\bk \subset \Omega.
\end{equation}
Consider the family of perforated domains
$$\Omega^\eps=\Omega\setminus \overline{T^\eps}.$$
\begin{figure}\label{figure:Omegaeps}
  \includegraphics[width=9cm]{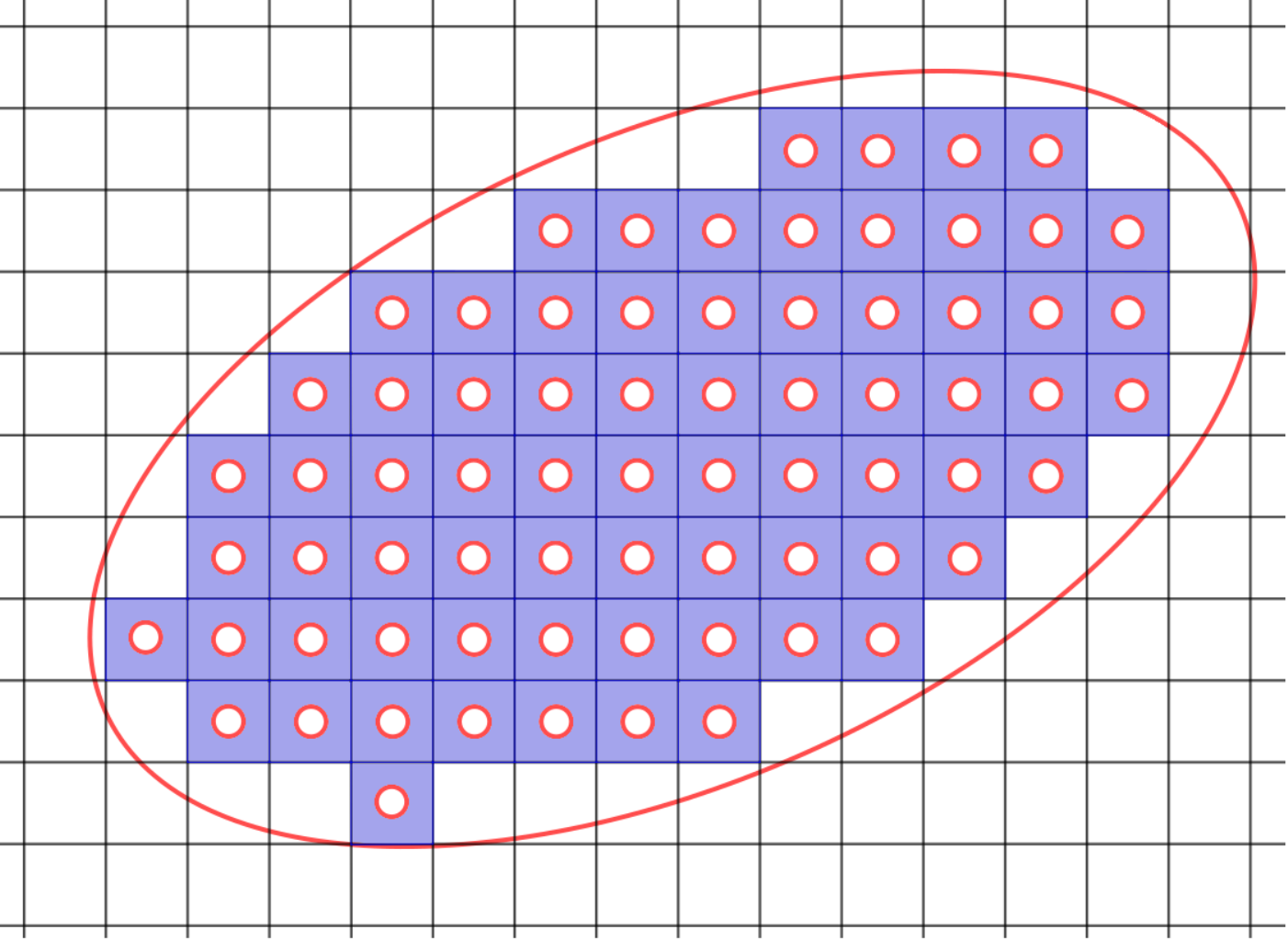}
  \caption{The domain $\Omega_\eps$}
\end{figure}
The Steklov eigenvalues of $\Omega^\eps$ are written as
$\sigma_k^\eps:=\sigma_k(\Omega^\eps),$ and we write $\set{u_k^{\eps}}$ for a corresponding complete sequence of eigenfunctions, normalized by
$$\int_{\partial\Omega^\eps}(u_k^{\eps})^2\de A=1.$$
Our first main result is the following critical regime homogenisation theorem
for the Steklov problem.
\begin{theorem}\label{thm:homogenisation}
  Suppose that $r_\eps^{d-1} \eps^{-d} \to \beta$ for some $\beta\in [0,\infty)$, as $\eps\searrow 0$.
    Then $\sigma_k^{\eps}$ converges
    to the eigenvalue $\Sigma_{k,\beta}$ of \eqref{problem:homo}.
    The functions $U_k^{\eps}\in \RH^1(\Omega)$ obtained by harmonic extension of a normalized Steklov
    eigenfunction $u_k^{\eps}$ over the holes $T^\eps\subset\Omega$ 
    form a sequence which weakly converges in $\RH^1(\Omega)$ to a solution $U_k$ associated with
    $\Sigma_{k,\beta}$ of \eqref{problem:homo}.
\end{theorem}

\begin{remark} \label{rem:multiple}
  If the eigenvalue $\Sigma_k=\Sigma_{k,\beta}$ is multiple of multiplicity $m$, \emph{i.e.}
  $$\Sigma_{k-1} < \Sigma_k = \dotso = \Sigma_{k+m-1} <
  \Sigma_{k+m},$$
  the convergence statement in the previous theorem is understood in the following
  sense. Given a basis $U_k,\dotsc,U_{k+m-1}$ for the eigenspace associated with $\Sigma_k$, there is a family of $m\times m$ orthogonal matrices $M(\eps)$ such that
  \begin{equation}
   M(\eps) \begin{pmatrix}
            U_k^\eps \\
            \vdots \\
            U_{k+m-1}^\eps
           \end{pmatrix}
\longrightarrow \begin{pmatrix}
            U_k \\
            \vdots \\
            U_{k+m-1}
           \end{pmatrix}
  \end{equation}
as $\eps \to 0$. One could also be content with the weaker statement that if the eigenvalues are multiple, the convergence statement of theorem \ref{thm:homogenisation} is only true up to taking a subsequence.
\end{remark}

\begin{remark}
  Literature on homogenisation theory is often concerned with the situation where holes are proportional to their reference cell. That is, $r_\eps=c\eps$ for some constant $c\in (0,1/2)$. 
  In this case one has $r_\eps^{d-1} \eps^{-d}\to\infty$. It follows
  from~\cite{ceg2} that $\sigma_k^\eps\to 0$.
  Indeed it is proved there that any bounded domain $\Omega \subset \R^d$ satisfies 
\begin{equation}
  \sigma_k(\Omega)\abs{\del\Omega}^{\frac 1{d-1}} \le C_{d,k},
\end{equation}
where the number $C_{d,k}>0$ depend only on the dimension $d$ and index $k$.
The hypothesis that $r_\eps^{d-1} \eps^{-d}\to\infty$ implies that
$\abs{\del\Omega^\eps} \to \infty$, which forces $\sigma_k^\eps \to 0$, as claimed. Note that this also corresponds to the homogenisation regime which was studied by Vanninathan in~\cite{vanninathan} for a slightly different problem, for which the Dirichlet boundary condition was imposed on $\partial\Omega$ and the Steklov condition on $\partial T^\eps$.

The regime that we consider in Theorem \ref{thm:homogenisation} is the critical
regime for the Steklov problem, where we observe a change of behaviour in the limiting problem. This is akin to the situation studied by Rauch--Taylor~\cite{RauchTaylor} and Cioranescu--Murat \cite{CioranescuMurat}.
\end{remark}

\subsection{Convergence to the Neumann problem and spectral comparison theorems}
The $\beta$ parameter in Problem \eqref{problem:homo} can be interpreted as a relative weight between the interior of $\Omega$ and its boundary $\del \Omega$ for the behaviour of that problem, see Section \ref{sec:dynamical} for details on this interpretation. Our second main result is the following theorem, describing
the specific dependence on $\beta$ in \eqref{problem:homo}.

\begin{theorem}\label{thm:limitbeta}
  For each $k\in\N$, the eigenvalue $\Sigma_{k,\beta}$ depends continuously on $\beta\in[0,\infty)$ and satisfies
	\begin{gather*}
	\lim_{\beta\to\infty}{A_d}\beta \Sigma_{k,\beta} = \mu_k.
	\end{gather*}
	The eigenfunctions $\set{U_{k,\beta}}$ satisfy
\begin{equation}
 \beta^{1/2} U_{k,\beta} \to f_k
\end{equation}
weakly in $\RH^1(\Omega)$ as $\beta \to \infty$, where $f_k$ is the $k$th non--trivial Neumann eigenfunction.
\end{theorem}
\begin{remark}
 We make the second observation that this convergence \emph{cannot} be uniform in $k$, as that would contradict \cite[Theorem 4.4]{BelowFrancois}. 
\end{remark}
The relationships between isoperimetric type problems for the Neumann and Steklov eigenvalue problems have been investigated for the first few eigenvalues in~\cite{FreiLauW2018,FreiLauS2018} from the point of view of the Robin problem. Our methods also allow us to investigate the relationship between these isoperimetric problems for every eigenvalue rank $k$.

The combination of Theorem~\ref{thm:limitbeta} and Theorem~\ref{thm:homogenisation} allows the transfer of known bounds for Steklov eigenvalues to bounds for Neumann eigenvalues. For instance, we can combine these two theorems with~\cite[Theorem 1.3]{ceg2}, asserting that for bounded Euclidean domains with smooth boundary
\begin{equation} \label{eq:boundceg2}
\sigma_k(\Omega)|\partial\Omega|^{\frac{2}{d-1}} \le C(d) k^{2/d}.
\end{equation}
This leads to the following.
\begin{cor} \label{cor:boundd}
  The Neumann eigenvalues of a bounded domain $\Omega\subset\R^d$ satisfy
  \begin{gather} \label{eq:dimbound}
    \mu_k(\Omega)|\Omega|^{2/d}\leq C(d)k^{2/d},
  \end{gather}
  where the constant $C(d)$ is exactly that of\, \cite[Theorem 1.3]{ceg2}.
\end{cor}
\begin{remark}
  The existence of a constant depending only on the dimension in inequality \eqref{eq:dimbound} is already known. In fact, Kröger obtained a better constant in \cite{Kroger1992}. However, it follows from Corollary \ref{cor:boundd} that any improvement to the bound  
  \eqref{eq:boundceg2} will transfer to bounds on Neumann eigenvalues.
  \end{remark}
  One of the original motivation for this project was the study of the following quantity:
$$\hat \sigma_k^*:=\sup\left\{\sigma_k(\Omega)|\partial\Omega|\,:\,\Omega\subset\R^2 \text{ bounded with smooth boundary}\right\}.$$
  In dimension $d=2$, we are able to get a stronger version of Corollary
  \ref{cor:boundd} in the sense that we obtain a direct link between $\hat \sigma_k^*$ and
  $$\hat \mu_k^*:=\sup\left\{\mu_k(\Omega)|\Omega|\,:\,\Omega\subset\R^2 \text{ bounded with smooth boundary}\right\}.$$
  In that case, we obtain.
  \begin{theorem} \label{thm:neumannd=2}
   For $d = 2$ and every $k \in \N$,
   \begin{equation}
   \hat \mu_k^* \le \hat \sigma_k^*.
   \end{equation}
  \end{theorem}
  \begin{remark}
  From~\cite{fraschoen3} we have that
  $$\sigma_1(\Omega)|\partial\Omega|<4\pi.$$
  It follows from Theorem~\ref{thm:homogenisation} and Theorem~\ref{thm:limitbeta} that
  $$\mu_1(\Omega)|\Omega|\leq 4\pi.$$
  Of course this bound is already known. Indeed the optimal upper bound is given by the famous Szeg\H{o}--Weinberger theorem:
  $$\mu_1(\D)\pi\approx 3.39\pi.$$
  One can remark, however, that $4\pi$ corresponds to the P\'olya bound for the first Neumann eigenvalue. We can now also see that obtaining bounds of a similar type for $\sigma_k$ would also transfer to $\mu_k$.
\end{remark}

The previous discussion also yields the following corollary.
\begin{cor}
  $$\hat \sigma_1^*\geq \mu_1(\D)\times\pi\approx 3.39\pi.$$
\end{cor}

Indeed, by the Szego-Weinberger inequality, we have that
\begin{equation}
 \hat \sigma^*_1 \ge \pi \mu_1(\D) \approx 3.39 \pi.
\end{equation}
Furthermore, this number will be approached as close as desired in homogenisation sequence of pierced unit disks with large enough parameter $\beta$.  Of course, it is known from \cite{fraschoen2} that if one is allowed to optimise amongst surfaces rather than Euclidean domains, then there is a sequence of surfaces such that $\sigma_1(\Omega_n)\abs{\del \Omega_n} \to 4\pi$.
This leads to the natural conjecture
\begin{conjecture}
  $$\hat \sigma_1^* = \mu_1(\D)\times\pi\approx 3.39\pi.$$
\end{conjecture}
Note that the previous best known lower bound for $\hat \sigma_1^*$ was 
attained on some concentric annulus, whose first normalised Steklov eigenvalue is approximately $2.17 \pi$, see \cite{gpsurvey}. We also observe that a similar analysis yields that for any $\Omega \subset \R^2$ bounded with smooth boundary,
$$ \hat \sigma_k^* \ge \mu_k(\Omega) \abs{\Omega}.$$ In particular, it follows from Weyl's law for Neumann eigenvalues that there exists a sequence $a_k \sim 4\pi k$ such that
$$\hat \sigma_k^* \ge a_k.$$

\subsection{Discussion}
Homogenisation theory is a young branch of mathematics which star\-ted around
the 1960's. Its general goal is to describe macroscopic properties of materials
through their microscopic structure. To the best of our knowledge, the first
papers to study periodically perforated domains from a rigourous mathematical
point of view are those of Marchenko and Khruslov from the early 1960's
(e.g.~\cite{MarchenkoKhruslov1964}) leading to their influential
book~\cite{MarchenkoKhruslov} in 1974. The topic became widely known in the West
with the work of  Rauch and Taylor on the \emph{crushed ice
problem}~\cite{RauchTaylor} in 1975 and then with the publication in 1982
of~\cite{CioranescuMurat} by Cioranescu and Murat. Many of these early results
were concerned with the Poisson problem $\Delta u_\eps=f$ under Dirichlet
boundary conditions $u_\epsilon=0$ on $\partial\Omega_\epsilon$. The limiting
behaviour of the solution $u_\eps$ depends on the rate at which $r_\eps\searrow
0$. Three regimes are considered. If the size of the holes $r_\epsilon$ tends to
zero very fast, then in the limit the solutions tend to those of the Poisson
problem on the original domain $\Omega$, while if the size of the holes are big
enough the solutions tend to zero. The main interest comes from the critical
regime, in which case the solutions tend to solutions of a new elliptic
problem.

In this paper we are concerned with the much less studied Steklov spectral
problem \eqref{Problem:Steklov}. From the point of view of homogenisation
theory, this problem is atypical in the sense that the function spaces which
occur for different values of the parameter $\eps>0$ are not naturally related.
This means in particular that this problem does not yield to any of the usual
general frameworks used in homogenisation theory (such as \cite[Chapter
11]{JikovKozlovOlenik}). Nevertheless, several authors have considered
homogenisation for this problem, using ad hoc methods depending on the specifc
situation considered. The behaviour of Steklov eigenvalues under singular
perturbations such as the perforation of a single hole has also been studied
in~\cite{GryCristo2014, Nazarov2014}.

Our main inspiration for this work is the paper \cite{vanninathan}. Several
papers have also considered homogenisation in the situation where the Dirichlet
condition is imposed on the outer boundary while the Steklov condition is
considered only on the boundary of the holes
\cite{ChiadoPiatNazarovPiatnitski,Douanla}. In these papers the holes are
proportional to the size of the reference cell. The novelty of our
homogenisation result in the case of the Steklov problem is that consider holes
that are shrinking much faster than that, in a critical regime where the
limiting problem
is fundamentally different. They are in fact shrinking at the precise rate which
makes their total surface area (or perimeter in dimension $2$) is asymptotically
comparable with the volume of the domain. This is similar to the work of
Rauch--Taylor \cite{RauchTaylor} for the Neumann problem, Cioranescu--Murat
\cite{CioranescuMurat} for the Dirichlet problem and Kaizu for the Robin problem
\cite{Kaizu}.

The \emph{energy method} of Tartar (see
\cite[Section 1.3]{allaire} for an exposition) has been used extensively in the
study of homogenisation problems at critical regimes, see
\cite{CioranescuMurat,Kaizu} and more recently \cite{CheredDondlRosler}, where
they obtain norm-resolvant convergence for the Dirichlet, Neumann and Robin
problem. That method uses an auxiliary function, satisfying some
energy-minimising PDE in the fundamental cells, in order to derive convergence
of the problem in the weak formulation. The method, in its Robin or Neumann
form, is boundary-condition agnostic and as such is ill-suited for the Steklov
problem, where the normalisation is with respect to $\RL^2(\del \Omega^\eps)$.
Indeed, while the technique could be used to obtain some form of convergence, it
will not be able to transfer boundary estimates to $\RL^2(\Omega)$, and ensure
that the limit solution doesn't degenerate to the trivial one.

Nevertheless, one can interpret our technique as a variation on the energy
method, adapted for problems defined on the boundary. We are also using an auxiliary PDE in order
to derive convergence, but it does not stem from compensated compactness. The
main difference, however, is that we can deduce interior estimates from those on
the boundary of the periodic holes from our auxiliary problem, see Lemma
\ref{lem:normconv}.

 We note that in this paper we have chosen to consider only spherical holes.
In fact, it should be possible to consider more general convex holes
obtained as scaled copies $r_\varepsilon\omega$ of a fixed convex set $\omega$, as
it is done in classical homogenization literature. Here, convexity of the holes
would be required for $\RL^\infty$ estimates of the Steklov eigenfunctions, see
Lemma \ref{lem:tracebv}.

Since we are motivated by spectral questions, namely to explore a new link
between the Steklov problem and the Neumann problem, we decided to avoid the 
technical difficulties which would occur by considering
more general inclusions (which would probably provide a more complicated limit problem
with some kind of capacity of the holes and a polarization matrix).
This allows us to get the simpler dynamical eigenvalue problem (3) and emphasise
clearly the link with the the classical
Neumann eigenvalue problem by letting the parameter $\beta \to +\infty$.

\subsection{Structure of the proof and plan of the paper}

In Section~\ref{section:NotationsFunctionSpaces} we formally describe properties
of the various eigenvalue problems that we study as well as the functions spaces over which they are defined. While they are well-known, the notation used for all of them often collides. In that section, we fix notation once and for all for the remainder of the paper for definedness and ease of references. 

In Section~\ref{section:ComparisonTheorems} we study one of the main technical tools in this paper, properties of harmonic extensions of functions on annuli to the interior disk. Results are separated in two categories: those that rely on the fact that the functions satisfy a Robin-type boundary condition, and more 
general results that do not rely on such a thing. Most of our results will be obtained by considering the Fourier expansion of a function
$$
u(\rho,\btheta) = \sum_{\ell,m} a_{\ell}^m(\rho) Y_\ell^m(\btheta)
$$
in spherical harmonics, and obtaining our inequalities term by term for every $a_\ell^m$. 

Section \ref{sec:homogenisation} is the \emph{pièce de résistance} of this paper.
It is where we show Theorem \ref{thm:homogenisation} and the proof proceeds in many steps. We first prove that the family of harmonic extensions $U^\eps:=U_k^\eps$ \/ is bounded in the Sobolev space $H^1(\Omega)$ hence there
   exists a subsequence $\eps_n\searrow 0$ such that $U^{\eps_n}$ weakly converges to a function $U\in H^1(\Omega)$. This allows us to consider properties of the weak limit $U$, and of an associated limit $\Sigma$ to the eigenvalue sequence $\sigma_k^{\eps_n}$.
   
   It is then not so hard to show that, using the weak formulations of Problems \eqref{Problem:Steklov} and \eqref{problem:homo} that the limit of the homogenised Steklov problem contains terms corresponding to the limit dynamical eigenvalue problem, plus some spurious terms that must be shown to
   converge to zero. This is done by studying two representations of the weak formulation using Green's identity either towards the inside of the holes or the inside of the domains $\Omega^\eps$. The first one is used to show that
   the functionals that arise in the study of the homogenisation problem are
   uniformly bounded, hence we can use smooth test functions. In the second representation, we are therefore allowed to use smooth test functions, which allows us to recover convergence to zero of the spurious terms. A key step in this argument is to  understand the limit behaviour of an auxiliary homogenisation problem for the transmission eigenvalue problem (see Proposition~\ref{prop:limitproblem}).
   
   Once we have established convergence to a solution of Problem \ref{problem:homo}, we end up showing that the limit eigenpair $(\Sigma,U)$ does not degenerate to the trivial function. Using variational characterisation of eigenvalues and eigenfunctions, we can also show that we get complete spectral convergence, and that subsequences are not needed.
   
  Finally, in Section \ref{Section:LargeBeta} we show convergence to the Neumann
  problem as $\beta \to \infty$. The method is similar to the one used at the end
  of the previous section, but many of the inequalities are more subtle. We also
  show the comparison theorems between Steklov and Neumann eigenvalues in this section.
  
  In this paper, we use $c$ and $C$ to mean constants whose precise value is not important to our argument, and whose exact value may change from line to line. We use the notations $f = \bigo g$ and $f \ll g$ interchangeably to mean that there exists a constant $C$ such that $\abs{f(x)} \le Cg(x)$.
  
\section{\bf Notation and function spaces}
\label{section:NotationsFunctionSpaces}
Four different eigenvalue problems will be used. The goal of this section is to introduce them and fix the relevant notation.
Throughout the paper, we use real valued functions.

\subsection{The Steklov problem on $\Omega$ and $\Omega^\eps$}

Given a bounded domain $\Omega$ whose boundary $\del \Omega$ is smooth, the Dirichlet-to-Neumann 
operator (DtN map) $\Lambda$ acts on $C^\infty(\del\Omega)$ as
\begin{equation}
  \Lambda f = \del_\nu \hat f,
\end{equation}
where $\hat f$ is the harmonic extension of $f$ to the interior of $\Omega$. The DtN map is an
elliptic, positive, self--adjoint pseudodifferential operator of order $1$. Because $\partial\Omega$ is compact, it follows from
standard theory of such operators, see \emph{e.g} \cite{shubin}, that the eigenvalues form a non--negative unbounded sequence
$\set{\sigma_k : k \in \N_0}$ and that there exists an orthonormal basis of $(f_k)$ of $\RL^2(\del\Omega)$ such that
$\Lambda f_k=\sigma_kf_k$.
The  harmonic extensions $u_k=\hat f_k$ satisfy the Steklov problem
\begin{equation}
  \begin{cases}
   - \Delta u_k = 0 & \text{in } \Omega,\\
    \del_\nu u_k = \sigma_k u_k &\text{on } \del\Omega.
  \end{cases}
\end{equation}
In general, we use the same symbol $u_k$ for the function on $\Omega$ and for its trace on the boundary $\del\Omega$.
The eigenvalue sequence for $\Omega^\eps$ is
denoted $\set{\sigma_k^\eps : k \in \N_0)}$, with corresponding eigenfunctions $u_k^\eps$.
The
eigenfunctions $u_k$, respectively $u_k^\eps$, form an orthonormal basis with
respect to the inner products
\begin{equation}
  (f,g)_\del := \int_{\del \Omega} f g \de A,
\end{equation}
respectively
\begin{equation}
  (f,g)_{\del^\eps} := \int_{\del \Omega^\eps} f g \de A.
\end{equation}
The $k$-th nonzero eigenvalue $\sigma_k$ is characterised by
\begin{equation}
  \label{eq:varcharsigm}
  \sigma_k = \inf\set{\frac{\int_\Omega \abs{\nabla u}^2 \de \bx}{\int_{\del \Omega}u^2
  \de A}:u \in \RH^1(\Omega) \text{ and }  (u,u_j)_{\del}=0 \mbox{ for } 0 \le j < k}.
\end{equation}
The eigenvalues
$\sigma_k^\eps$ have the same characterisation, integrating over $\Omega^\eps$
and $\del \Omega^\eps$ respectively instead, and with the orthogonality being
with respect to $(\cdot,\cdot)_{\del^\eps}$.

\subsection{Dynamical eigenvalue problem} \label{sec:dynamical}
For $\beta \in (0,\infty)$, consider
the eigenvalue problem:
\begin{equation}\label{eq:dyneigenprob}
\begin{cases}-\Delta U = A_d \beta \Sigma U & \text{in } \Omega,\\
    \del_\nu U = \Sigma U & \text{on } \del \Omega.
\end{cases}
\end{equation}
where $A_d$ is the area of the unit sphere in $\R^d$.
Problem \eqref{eq:dyneigenprob} was introduced with a slightly different normalization in \cite{BelowFrancois}, where it is called a \emph{dynamical eigenvalue problem}.
The eigenvalues and
eigenfunctions are those of the operator
\begin{equation}
  P := \begin{pmatrix} -(A_d\beta)^{-1} \Delta & 0 \\ 0 & \del_\nu \end{pmatrix}.
\end{equation}
This unbounded operator is 
defined on an appropriate domain in the space $\RL^2_{A_d \beta}(\Omega) \times \RL^2 (\del \Omega)$ which consists simply of
$\RL^2(\Omega) \times \RL^2 (\del \Omega)$ equipped
with the inner product defined by
\begin{equation}
  (f,g)_\beta := A_d \beta\int_\Omega f g \de \bx + \int_{\del \Omega} f
  g \de A,
\end{equation}
The dynamical eigenvalue problem~\eqref{eq:dyneigenprob}
has a discrete sequence of eigenvalues
\begin{equation}
  0 = \Sigma_{0,\beta} < \Sigma_{1,\beta} \le \Sigma_{2,\beta} \le \dotso \nearrow
  \infty.
\end{equation}
Let $X\subset \RL^2_{A_d \beta}(\Omega) \times \RL^2 (\del \Omega)$ be the subspace
defined by
\begin{equation}
  X:= \set{U = (u,\tau u) : u \in \RH^1(\Omega)},
\end{equation}
where $\tau:\RH^1(\Omega) \to \RL^{2}(\del\Omega)$ is the trace operator.
The eigenfunctions $U_{k,\beta}$ associated to $\Sigma_{k,\beta}$ form a basis of $X$,
but not of $\RL^2_{A_d\beta}(\Omega) \times \RL^2(\partial\Omega)$. The eigenvalues
$\Sigma_{k,\beta}$ are characterised by

\begin{equation}
  \Sigma_{k,\beta} = \inf\set{\frac{\int_\Omega \abs{\nabla U}^2 \de
    \bx}{A_d \beta\int_\Omega U^2 \de \bx + \int_{\del \Omega} U^2
      \de A}:U \in \RH^1(\Omega) \text{ and }  (U,U_{j,\beta})_{\del}=0 \mbox{ for }0 \le j < k
    }.
\end{equation}

\subsection{The Neumann eigenvalue problem}
We will also make use of the classical Neumann eigenvalue problem
\begin{equation}
  \begin{cases}-\Delta f = \mu f & \text{in } \Omega,\\
    \del_\nu f = 0 &
    \text{on } \del \Omega.
  \end{cases}
\end{equation}
The Neumann eigenvalues form an increasing sequence
\begin{equation}
  0 = \mu_0 < \mu_1 \le \mu_2 \le \dotso \infty,
\end{equation}
with associated eigenfunctions $f_k$, orthonormal with respect to the
$\RL^2(\Omega)$ inner product
\begin{equation}
  (f,g) := \int_\Omega f g \de\bx.
\end{equation}
The eigenvalues are characterised by
\begin{equation}
  \mu_k = \inf\set{\frac{\int_\Omega \abs{\nabla f}^2\de \bx }{\int_{\Omega} f^2
  \de \bx}: f \in \RH^1(\Omega) \text{ and } (f,f_j)=0 \mbox{ for }j=0,1,\cdots,k-1}. 
\end{equation}

\section{\bf Comparison theorems}\label{section:ComparisonTheorems}
In this section, we derive comparison inequalities that will be used repeatedly.
For $0 < r < R$, set $$A_{r,R} = B(0,R) \setminus\overline{B(0,r)}.$$
For any set $X$ where it is defined, the Dirichlet energy of a function $f:X\rightarrow\R$ is
$$\CD(f) :=\int_{X}|\nabla f|^2\,\de\bx.$$
 The Dirichlet energy of $f$ on a subset $Y\subset X$ is written
$\CD(f;Y)$. Note that since every problem under consideration is self-adjoint no
generality is lost by studying real-valued functions.

Lemmas \ref{lemma:localEstimateFourier} and \ref{lemma:Sobolevconst} are proved by
  using a Fourier series decomposition for functions in an annulus.
  We note that even though the proofs rely on estimates on the radial part of
  these functions, we do not claim that the inequalities proved therein are
  realised by radial optimisers. The standard Schwarz symmetrisation arguments do not
apply here because the domain of the functions are annuli rather than balls.

\subsection{Comparison theorems for functions satisfying a Steklov boundary
condition}

\begin{lemma}\label{lemma:localEstimateFourier}
  Fix a postive real number $\sigma>0$. For any $0 < r < R \le 1$,
  let $u\in C^\infty(\overline{A_{r,R}})$ be such that
  \begin{equation}
  \begin{cases}
  \Delta u=0&\mbox{ in }A_{r,R},\\
  \partial_\nu u=\sigma u&\mbox{ on } \partial B(0,r).
  \end{cases}
  \end{equation}
  Consider the function $h:B(0,r)\rightarrow\R$ defined by
  $$\begin{cases}
  h=u&\mbox{ on }\partial B(0,r),\\
  \Delta h=0&\mbox{ in }B(0,r).
  \end{cases}
  $$
  Then as the ratio $r/R$ goes to $0$,
  \begin{equation}
    \label{eq:dircomp}
    \CD(h) \le 5 \CD(u) \left(\frac r R \right)^d \left( 1 + \bigo{\left(\frac r
    R \right)^d} \right).
  \end{equation}
  \end{lemma}

\begin{proof}
  For every $\ell \ge 0$, we denote by $N_\ell$ the dimension of the space
  $H_\ell$ of spherical harmonics of order $\ell$ and denote by $Y_\ell^m(\btheta)$, 
  $1 \le m \le N_\ell$ the standard orthonormal basis of spherical harmonics on the unit sphere.
  On $A_{r,R}$, the function $u$ admits a Fourier decomposition in spherical harmonics
\begin{equation}
  u(\rho,\btheta) = \sum_{\substack{\ell \ge 0 \\1 \le m \le N_\ell}}
  a_\ell^m(\rho)
  Y_\ell^m(\btheta).
  \label{eq:fourierspherical}
\end{equation}
We start by studying the form of the coefficients $a_\ell^m$.
\newline
\noindent
\textbf{Case $d>2$}.
The harmonicity condition on $u$ implies that the radial parts
$a_\ell^m(\rho)$ are given by
\begin{equation}
  a_\ell^m(\rho) = c_\ell^m \rho^{\ell} + c_{-\ell}^m \rho^{-\ell + 2 - d}.
  \label{eq:radial}
\end{equation}
By convention the coefficients $c_0^1$ and $c_{-0}^1$ are assumed to be different, the minus sign referring as for the other coefficients to the solution blowing up at the origin. The Steklov condition for $u$ on $\partial B(0,r)$ along with the orthogonality of
the spherical harmonics $Y_\ell^m$ imply
\begin{equation}
  - \del_\rho a_\ell^m(r) = \sigma a_\ell^m(r),
  \label{eq:steklovint}
\end{equation}
which yields the relations
\begin{equation}
  c_{-\ell}^m = \overbrace{\left(\frac{\ell + r \sigma}{\ell - 2 + d - r
  \sigma}\right)}^{M:=}r^{2\ell+d-2}c_\ell^m.
  \label{eq:relcoeff}
\end{equation}
In turns this yields the following explicit expression for the radial functions
\begin{equation}
  \label{eq:radexplicit}
  a_\ell^m(\rho) = c_\ell^m \rho^{\ell}\left(1+ M
  \left(\frac{r}{\rho}\right)^{2\ell+d-2}\right).
\end{equation}
For $r \le \frac{d - 2}{2\sigma}$, it follows that
\begin{equation}
\label{eq:Mdg2}
  \frac{\ell}{\ell - 2 + d}\le  M \le 1.
\end{equation}

  \noindent
\textbf{Case $d=2$}.
Equations~\eqref{eq:radial} holds except at $\ell = 0$,
in which case, 
$$a_0^1(\rho) = c_0^1 + c_{-0}^1 \log \rho.$$
In that case, \eqref{eq:radexplicit} becomes
\begin{align*}
 a_0^1 &= c_0^1 \left(1 + \frac{r\sigma}{1 + r \log r} \log(1/\rho) \right) \\
 &= c_0^1 \left(1 + M \log(1/\rho)\right).
 \end{align*}
Observe for the sequel that when $d = 2$ and $\ell = 0$, then $M = \bigo{r\sigma}$, and if $\ell > 0$, then
\begin{equation} \label{eq:Md2}
  M = 1 + \bigo{r \sigma}.
\end{equation}

 Because inequality~\eqref{eq:dircomp} is invariant
under scaling, it is sufficient to prove the case $R= 1$ and to let $r \to 0$.
The harmonic extension of $u$ to $B(0,r)$ is given by
\begin{equation}
  h(\rho,\btheta) = \sum_{\substack{\ell \ge 0 \\ 1 \le m \le N_\ell}}
  a_\ell^m(r) \frac{\rho^\ell}{r^\ell} Y_\ell^m(\btheta).
  \label{eq:extension}
\end{equation}
The Dirichlet energy of $h$ is
\begin{equation}
  \CD(h) = \sum_{\substack{\ell \ge 1 \\ 1 \le m \le N_\ell}} \ell a_\ell^m(r)^2
  r^{d-2}.
  \label{eq:dirichh}
\end{equation}
On the other hand, the Dirichlet energy of $u$ is given, from Green's identity,
and the Steklov condition on $\partial B(0,r)$
\begin{equation}
  \CD(u) = \sum_{\ell,m} \sigma a_\ell^m(r)^2r^{d-1} + \sum_{\ell,m}
  a_\ell^m(1)\del_\rho a_\ell^m(1).
  \label{eq:dirichu}
\end{equation}
Our goal is now to find a bound on \eqref{eq:dirichh} in terms of
\eqref{eq:dirichu}.
It is sufficient to show that for each $\ell \ge 1$,
\begin{equation}
  \ell a_\ell^m(r)^2 \le 4 r^2 a_\ell^m(1) \del_\rho a_\ell^m(1)\left( 1 +
  \bigo{r^d}    \right).
  \label{eq:sufficientbound}
\end{equation}
Suppose without loss of generality that $c_\ell^m = 1$. The substitution of \eqref{eq:relcoeff} into \eqref{eq:radial} imply that
\begin{equation}
   a_\ell^m(r)^2 =  r^{2 \ell} \left( 1 + M \right)^2 
  \le 5r^2,
  \label{eq:coeffh}
\end{equation}
and that
\begin{equation}
  a_\ell^m(1)\del_\rho a_\ell^m(1) = \ell \left( 1 + \frac{(2-d)
    M}{\ell}
r^{2\ell + d - 2} + \frac{( 2 - d - \ell) M^2}{\ell} r^{2(2 \ell + d-2)}\right)
= \ell + \bigo{r^d}.
  \label{eq:coeffu}
\end{equation}
Hence, dividing \eqref{eq:coeffh} by \eqref{eq:coeffu}, and using the bounds
on $M$ from \eqref{eq:Mdg2} and \eqref{eq:Md2} we have that for $\ell\ge 1$,
\begin{equation}
\ell a_\ell^m(r)^2 \le 5 a_\ell^m(1) \del_\rho
a_\ell^m(1)r^2 \left( 1 + \bigo{r^d} \right).
  \label{eq:boundcoeffs}
\end{equation}
This proves  our claim.

\end{proof}

\subsection{General $\RH^1$ comparison theorems on annuli and balls}

The next two lemmas do not depend on any  specific boundary condition.
The first one gives bounds for Sobolev constants of annuli.
\begin{lemma}
  \label{lemma:Sobolevconst}
  For \/ $0 < r < R < 1$, define
  \begin{equation}
    \label{eq:rayleighgamma}
    \gamma(r,R) := \inf\set{
      \frac{\int_{A_{r,R}} \abs{\nabla u}^2 + u^2\,\de \bx}{\int_{\del
    B(0,r)} u^2\,dA}:u \in \RH^1(A_{r,R}), u\big|_{\del B(0,r)} \not \equiv 0}.
  \end{equation}
 Suppose that $R \ge c r^{\frac{d-1}{d}} \ge 2r$ for some $c > 0$. Then, there
  is a constant $C$ depending only on the dimension and on $c$ such that
  \begin{equation}
    \gamma(r,R) \ge C \min\set{R^d r^{1-d},r^{\frac 1 d - 1} }.
  \end{equation}
\end{lemma}

\begin{proof}
 
  As earlier, write a function $u \in \RH^1(A_{r,R})$ as
\begin{equation}
  u(\rho,\btheta) = \sum_{\ell,m} a_\ell^m(\rho) Y_\ell^m(\btheta).
\end{equation}
Using the notation $u_{\btheta}$ for the tangential gradient, the Dirichlet energy of $u$ is expressed as
\begin{equation}
  \begin{aligned}
    \CD(u) &= \int_r^R \int_{\mathbb S^{d-1}} \left( u_\rho^2 + \rho^{-2} u_{\btheta}^2 \right)
    \rho^{d-1} \de \btheta \de \rho \\
    &\geq
    \int_r^R \int_{\mathbb S^{d-1}} \left( u_\rho^2\right)
    \rho^{d-1} \de \btheta \de \rho \\
    &= \sum_{\ell,m} \int_r^R \left(a_\ell^m(\rho)'\right)^2 \rho^{d-1} \de
    \rho.
  \end{aligned}
\end{equation}
On the other hand, the denominator in \eqref{eq:rayleighgamma} is given by
    \begin{equation}
      \int_{\del B(0,r)} u^2 \de A = r^{d-1} \sum_{\ell,m} a_\ell^m(r)^2.
    \end{equation}
    Combining these last two expressions in \eqref{eq:rayleighgamma} and defining the
    density $w(\rho) = \left(\frac \rho r \right)^{d-1}$, we see it is enough to prove that
\begin{equation} 
  \label{eq:ineqsum}
  \sum_{\ell,m} \int_r^R \left(\left( \del_\rho a_\ell^m(\rho) \right)^2 +
  a_\ell^m(\rho)^2\right)
  w(\rho) \de \rho \ge  C \min\set{R^d r^{1-d},r^{\frac 1 d - 1} } \sum_{\ell,m}
  a_\ell^m(r)^2.
\end{equation}
Indeed, working term by term, we prove that any smooth function
$f : [r,R] \to \R$ satisfies
\begin{equation}
  \int_r^R \left(f'(\rho)^2 + f(\rho)^2\right)w(\rho) \de \rho \ge
  C \min\set{R^d r^{1-d},r^{\frac 1 d - 1} } f(r)^2.
\end{equation}
To this end, assume without loss of generality that $f(r) = 1$.
Following a strategy that was used in \cite{CEG3} and in \cite{CianciGirouard}, consider the two following situations.

Let $t\in (r,R)$, to be fixed later.

\noindent
\textbf{Case a.} Suppose first that for all $\rho \in (t,R)$,
$\abs{f(\rho)} \ge 1/2$. It follows from monotonicity and explicit integration that
\begin{equation}
  \int_r^R \abs{f(\rho)}^2 w(\rho) \de \rho \ge 
  \frac{r^{1-d}}{4d}\left(R^d - t^d \right).
\end{equation}

\noindent
\textbf{Case b.} Suppose there exists $\rho_0 \in (t,R)$ such that
$f(\rho_0) < 1/2$. Splitting the integral, using the fact that $w(\rho) \ge 1$ and is increasing
for all $\rho$ together with the Cauchy--Schwarz inequality leads to
\begin{equation}
  \label{eq:sumsquares}
  \begin{aligned}
    \int_r^R f'(\rho)^2 w(\rho) \de \rho &\ge \int_r^t f'(\rho)^2 +
    \int_t^{\rho_0} f'(\rho)^2w(\rho) \de \rho \\
    &\ge \frac{1}{t - r}\left( \int_r^t f'(\rho) \de \rho \right)^2 +
    \left(\frac t r \right)^{d-1}\frac{1}{\rho_0 - t}\left(\int_t^{\rho_0} f'(\rho) \de
    \rho\right)^2.
  \end{aligned}
\end{equation}
  By hypothesis $R<1$ so that $\frac{1}{\rho_0-t}>1$. This leads to
\begin{equation}
  \begin{aligned}
    \int_r^R f'(\rho)^2 w(\rho) \de \rho 
    \ge \frac 1 2 \min\set{\frac{1}{t - r},\left( \frac t r
    \right)^{d-1}}\left(\int_r^{\rho_0}f'(\rho) \de \rho\right)^2 ,
  \end{aligned}
\end{equation}
Choosing $t = \min\set{r^{\frac{d-1}{d}},R/2}$ guarantees that
  $\min\set{\frac{1}{t - r},\left( \frac t r
    \right)^{d-1}}=\left( \frac t r
  \right)^{d-1}$ so that
\begin{equation}
  \begin{aligned}
    \int_r^R \abs{f'(\rho)}^2 w(\rho) \de \rho 
    \ge \frac{1}{2}\left( \frac t r
    \right)^{d-1}\left(\int_r^{\rho_0}f'(\rho) \de \rho\right)^2.
  \end{aligned}
\end{equation}
It follows from the definition of $t$ that
$$\left(\frac{t}{r}\right)^{d-1}\geq\overbrace{\min\left\{1,\left(\frac{c}{2}\right)^{d-1}\right\}}^{A:=}r^{\frac{1-d}{d}}.$$
We can bound asymptotically the last integral in \eqref{eq:sumsquares} as
\begin{equation}
  \begin{aligned}
    \int_r^R f'(\rho)^2 w(\rho) \de \rho
    & \geq Cr^{\frac{1-d}{d}}
  \left(\int_r^{\rho_0}f'(\rho)\de \rho\right)^2\\
  &\geq \frac{A}{4}r^{\frac{1-d}{d}}.
\end{aligned}
\end{equation}
This also ensures that $R^d - t^d \ge (1-2^{-d})R^d$.
Since both situations are exclusive, inequality \eqref{eq:ineqsum} holds, finishing the proof.
\end{proof}
The next lemma compares $\RL^2$ norms on $B(0,r)$ with $\RH^1$ norms on
$B(0,R)$. Here, for any $\Omega' \subset \Omega$, the norm on $\RH^1(\Omega')$ is given by
\begin{equation}
  \norm{u}_{\RH^1(\Omega')}^2 = \norm{u}_{\RL^2(\Omega')}^2 + \norm{\nabla
  u}^2_{\RL^2(\Omega')^2}.
\end{equation}
\begin{lemma} \label{lemma:smalltolarge}
 For \/ $0 < r < R\leq 1$, if $R \ge c r^{\frac{d-1}{d}}$ for some $c>0$, there is a constant $C$ depending only on $c$ and on the dimension such that for all $u \in \RH^1(B(0,R))$, 
\begin{equation} \label{eq:boundrR}
 \norm{u}_{\RL^2(B(0,r))} \le C r^{1/2} \norm{u}_{\RH^1(B(0,R))}.
\end{equation}
 \end{lemma}
\begin{proof}
   Let $u\in H^1(B(0,R))$. Given $r\in (0,R)$,
   $$\int_{B(0,r)}u^2\,\de \bx=\int_{0}^r\rho^{d-1}\int_{S^{d-1}}u^2\,\de\theta\,\de\rho
   =\int_{0}^r\|u\|_{L^2(\partial B(0,\rho))}^2\,\de\rho
   $$
   It follows from the definition of $\gamma$ in Lemma~\ref{lemma:Sobolevconst} that
   $$
   \|u\|_{\partial B(0,\rho)}^2\leq\frac{1}{\gamma(\rho,R)}\|u\|_{H^1(A_{\rho,R})}^2
   \leq\frac{1}{\gamma(\rho,R)}\|u\|_{H^1(B(0,R))}^2.
   $$
   Substitution in the above leads to
   $$\int_{B(0,r)}u^2\,\de \bx\leq
   \|u\|_{H^1(B(0,R))}^2\int_{0}^r\frac{1}{\gamma(\rho,R)} \de \rho.
   $$
   It follows from Lemma~\ref{lemma:Sobolevconst} that
     \begin{align*}
       \int_{0}^r\frac{1}{\gamma(\rho,R)}\,\de\rho&\leq\frac{1}{C}\int_{0}^r\left(\frac{\rho^{d-1}}{R^d}+\rho^{1-\frac{1}{d}}\right)\de\rho\\
       &=\frac{1}{C}\left(\frac{r^d}{dR^d}+\frac{r^{2-\frac{1}{d}}}{2-\frac{1}{d}}\right)\\
       &\leq
       \frac{1}{C}\left(\frac{r}{cd}+\frac{r^{2-\frac{1}{d}}}{2-\frac{1}{d}}\right)\qquad\longleftarrow\text{ since }R^d\geq cr^{d-1}\\
       &\leq \tilde{C}r.
     \end{align*}
\end{proof} 
Finally, we require the following lemma about the behaviour of the boundary trace operator as a domain gets shrunk.

\begin{lemma}
 \label{lem:traceconstant}
 Let $\Omega \subset \R^d$ be a bounded open set and denote
 \begin{equation}
  \gamma(\Omega):= \inf\set{\frac{\int_\Omega \abs{\nabla u}^2 + u^2 \de \bx}{\int_{\del \Omega} u^2 \de A} : u \in \RH^1(\Omega), u\big|_{\del \Omega} \not \equiv 0}.
 \end{equation}
 Then, the following inequality holds as $\eps \to 0$:
 \begin{equation}
  \gamma(\eps\Omega) \ge \eps \gamma(\Omega).
 \end{equation}
\end{lemma}
\begin{proof}
 Consider $u \in \RH^1(\eps\Omega)$, $u\big|_{\del \eps\Omega} \not \equiv 0$. Then, with the change of variable $\by = \eps^{-1} \bx$,
 \begin{align}
  \int_{\eps \Omega} \abs{\nabla u(\bx)}^2 + u(\bx)^2 \de \bx &=
  \eps^{d} \int_{\Omega} \eps^{-2} \abs{\nabla u(\by)}^2 + u(\by)^2 \de \by \\
  &\ge \eps^{d} \int_{\Omega} \gamma(\Omega) \int_{\del \Omega} u(\by)^2 \de A 
  \\
&= \eps \gamma(\Omega) \int_{\del \eps \Omega} u(\bx)^2 \de A.
 \end{align}

\end{proof}

\subsection{Uniform bounds on Steklov eigenfunctions}

In order to obtain convergence of the eigenfunctions we will need that their
$\RL^\infty$ norm stays bounded. In this subsection, since we will need to
understand specific interplay between boundary surface area and volume, we will diverge
from our convention and denote the area of a boundary $\del \Omega$ by
$\CH^{d-1}(\del \Omega)$. In \cite{bgt}[Theorem 3.1], the authors prove that for any
Steklov eigenfunction $u$ of a domain $\Omega$ with eigenvalue $\sigma$,
\begin{equation}
  \norm{u}_{\RL^\infty(\Omega)} \le C \norm{u}_{\RL^2(\del \Omega)}
\end{equation}
where $C$ depends continuously on the dimension $d$, $\sigma$, $\abs{\Omega}$
and the norm of the trace application
\begin{equation}
  T : \RB\RV(\Omega) \to \RL^1(\del \Omega).
\end{equation}
It is clear that in our situation, the dimension and $\abs{\Omega}$ stay
bounded. The eigenvalues will be shown later to also stay bounded, but we use
to control the norm of $T$. In \cite{bgt}[Proposition 5.1], they give the
following
condition under which said norm stays bounded for a family of domains. If
$\set{\Omega^\eps}$ is a family of open bounded domains with
$\CH^{d-1}(\del\Omega^\eps) < \infty$ and such that $\Omega^\eps \subset K$ for some
bounded set $K$, and if there exists constants $Q$ and $\delta$ such that for every
$x \in \del \Omega^\eps$
\begin{equation}
  \label{eq:tracebound}
  \sup \set{\frac{\CH^{d-1}(\del^* E \cap \del^* \Omega^\eps)}{\CH^{d-1}(\del^*
    E \cap \Omega^\eps)} : E \subset \Omega^\eps \cap B_\delta(x),
  \operatorname{Per}(E,\Omega_n) < \infty} \le Q,
\end{equation}
then the norm of $T : \RB\RV(\Omega^\eps) \to \RL^1(\del \Omega^\eps)$ is uniformly bounded in $\eps$. Here, $\del^* E$
denotes the reduced bondary of $E$, which in general may be much smaller than
the topological boundary.

The next lemma is inspired by \cite{bgt}[Example 2]. Note that in their example,
$r_\eps = \smallo{\eps^{\frac{2d-1}{d-1}}}$, which means that the radius of the
holes if one order of magnitude in $\eps$ smaller than
the critical level at which our holes are going to $0$. Parts of the proof are in a similar spirit but we need more
precise estimates separately around every boundary component.

\begin{lemma}
  \label{lem:tracebv}
  For all $\eps>0$ sufficiently small, the norm of $T : \RB\RV(\Omega^\eps) \to
  \RL^1(\del\Omega^\eps)$ is uniformly bounded in $\eps$. 
\end{lemma}

\begin{proof}
  Following \eqref{eq:tracebound}, for any $x_\eps \in \del \Omega^\eps$, we
  want to give a uniform upper bound for the ratio
  \begin{equation}
    \label{eq:ratiobounded}
    \frac{\CH^{d-1}(\del^* E \cap \del \Omega^\eps)}{\CH^{d-1}(\del^* E \cap
    \Omega^\eps)}
  \end{equation}
  for sets $E \subset B_\delta(x_\eps) \cap \Omega^\eps$. Let us make the
  observation that for $\eps>0$ small enough,
  \begin{equation}
    \#\set{\bn \in I^\eps : Q_\bn^\eps \cap B_\delta(x_\eps) \ne \emptyset} \le
    2\omega_d\left(\frac \delta \eps\right)^d
  \end{equation}
  and for some $M$, $\CH^{d-1}(\del \Omega \cap \overline{B_\delta(x_\eps)}) \le
  M\delta^{d-1}$. On one hand, setting $\tilde \beta = \max(\beta,1)$, for any $E \subset \Omega^\eps \cap B_\delta(x_\eps)$
    of finite perimeter,
  \begin{equation}
    \label{eq:choicedelta}
    \CH^{d-1}(\del^* E \cap \del \Omega^\eps) \le \CH^{d-1}(\del \Omega^\eps
    \cap \overline{B_\delta(x_\eps)}) \le \left(M + 2d\tilde \beta
      \delta\right) \delta^{d-1}.
  \end{equation}
  On the other hand, we now need to find a lower bound for the denominator in
  \eqref{eq:ratiobounded} terms of the numerator. 
  By definition of $I^\eps$, for all $\bn \in I^\eps$ we have that $\del
  \Omega^\eps \cap Q_\bn^\eps = \del T_\bn^\eps$. From this, we can decompose
  \begin{equation}
    \label{eq:decompositionboundary}
    \CH^{d-1}(\del^*E \cap \del \Omega^\eps) = \CH^{d-1}(\del^* E \cap \del
    \Omega) + \sum_{\bn \in I^\eps} \CH^{d-1}(\del^*E \cap \del T_\bn^\eps).
  \end{equation}
  We first observe that if $\delta$ is chosen such that $\omega_d \delta^d \le
  \abs{\Omega}/2$ and $\CH^d(E)\leq 1$, the trace inequality for $BV(\Omega)\rightarrow L^1(\partial\Omega)$ and the relative isoperimetric inequality relative to $\Omega$
  \cite{bgt}[Inequality (2.2)]  applied to the characteristic function
  $\chi_E\in BV(\Omega)$  implies that there is a constant $C$, whose precise
  value may change from line to line but which depens only on $\Omega$, such that
  \begin{equation}
    \label{eq:extbdry}
    \begin{aligned}
      \CH^{d-1}(\del^*E \cap \del \Omega)
      &\le C \CH^{d-1}(\del^*E \cap \Omega)\\& = C
      (\CH^{d-1}(\del^*E \cap (\del \Omega^\eps \setminus \del \Omega)) +
      \CH^{d-1}(\del^*E \cap \Omega^\eps))\\
      &=
      C(\sum_{\bn \in I^\eps} \CH^{d-1}(\del^*E \cap \del T_\bn^\eps) +
      \CH^{d-1}(\del^*E \cap \Omega^\eps))
    \end{aligned}
  \end{equation}
  Substitution in~\eqref{eq:decompositionboundary} implies
  \begin{equation}\label{eq:extbdry2}
    \CH^{d-1}(\del^*E \cap \del \Omega^\eps) \leq C\left(\CH^{d-1}(\del^* E \cap \Omega^\eps) + \sum_{\bn \in I^\eps} \CH^{d-1}(\del^*E \cap \del T_\bn^\eps)\right).
  \end{equation}  
  For $\bn \in I^\eps$ and $t\in (0,\eps/4)$, define
  \begin{equation}
    F_{\bn,t} := \{x\in E\,:\,\dist_{\del T_\bn^\eps}(x) \leq t\}\subset Q_\bn^\eps.
  \end{equation}  
  Assume that there is $t \in (0,\eps/4)$ such that
  \begin{equation}
    \label{eq:localisedboundary}
    \CH^{d-1}(\del^* F_{\bn,t} \cap Q_{\bn}^\eps) \le 2\CH^{d-1}(\del^* E \cap
    Q_\bn^\eps).
  \end{equation}
  Since projections on convex sets are nonexpansive and $F_{\bn,t} \subset
  Q_{\bn}^\eps$, we have that
  \begin{equation}
    \label{eq:locbdry}
    \begin{aligned}
      \CH^{d-1}(\del^* E \cap \del T_\bn^\eps) &= \CH^{d-1}(\del^* F_{\bn,t} \cap \del
      T_\bn^\eps) \\
      &\le \CH^{d-1}(\del^* F_{\bn,t} \cap Q_{\bn}^\eps) \\
      &\le 2\CH^{d-1}(\del^*E \cap  Q_{\bn}^\eps).
    \end{aligned}
  \end{equation}
  If \eqref{eq:localisedboundary} holds for each $\bn\in I^\eps$ then it follows from \eqref{eq:extbdry2} that
  \begin{equation}
    \begin{aligned}
      \CH^{d-1}(\del^*E \cap \del \Omega^\eps) &\leq C\left(\CH^{d-1}(\del^* E \cap \Omega^\eps) + \sum_{\bn \in I^\eps}\CH^{d-1}(\del^*E \cap  Q_{\bn}^\eps) \right)\\
      &\leq C\CH^{d-1}(\del^* E \cap \Omega^\eps),
    \end{aligned}
  \end{equation}
  which completes the proof in this case.
  
  Otherwise, let $J^\eps := \set{\bn \in I^\eps : \text{equation
      \eqref{eq:localisedboundary} does not hold}}$.
  For $\bn\in J^\eps$, set
  \begin{equation}
    h_\bn(t) := \CH^{d-1}\left(\{x\in \del^*F_{\bn,t}\,:\, \dist_{\del T_\bn^\eps}(x) =
    t\}\right).
  \end{equation}
  Since \eqref{eq:localisedboundary} does not hold,
  $$2h_\bn(t)\geq \CH^{d-1}(\del^*F_{\bn,t} \cap Q_\bn^\eps).$$
  It follows from the
  relative isoperimetric inequality with respect to $Q_\bn^\eps$  that
  \begin{equation}
      \begin{aligned}
        c\CH^d(F_{\bn,t})^{\frac{d-1}{d}}
        &\le \CH^{d-1}(\del^*F_{\bn,t} \cap Q_\bn^\eps) \le 2 h_\bn(t).
      \end{aligned}
    \end{equation}
  The coarea formula gives $\del_t \CH^{d}(F_{\bn,t}) = h_\bn(t)$, and it follows by integration and the relative isoperimetric inequality with respect to $Q_\bn^\eps$ that
  \begin{equation}
    \label{eq:lowerbound}
    \begin{aligned}
      d\CH^d(F_{\bn,\eps/4})^{1/d} &= \int_0^{\eps/4} \frac{h_\bn(t)}{\CH^{d}(F_{\bn,t})^{\frac{d-1}{d}}} \de t 
      \ge C \eps.
    \end{aligned}
  \end{equation}
  That is for $\bn\in J^\eps$,
  \begin{equation}
    \label{eq:firstpart}
  \begin{aligned}
    \CH^d(E\cap Q_\bn^\eps)&\geq\CH^d(F_{\bn,\eps/4})\\
    &\geq C\eps^d\\
    &=\CH^d(Q_\bn^\eps).
  \end{aligned}
  \end{equation}
  This implies that
  for $\eps>0$ small enough one has
  \begin{equation}
    \begin{aligned}
      \CH^d(E\cap Q_\bn^\eps)
      &\geq\frac{C}{\tilde \beta}\CH^{d-1}(\partial T_\bn^\eps).
    \end{aligned}
  \end{equation}  
  The isoperimetric inequality for $E\subset\R^d$ gives
  \begin{equation}
    \begin{aligned}
      \sum_{\bn \in J^\eps}\CH^{d-1}(\partial T_\bn^\eps)&\leq\tilde \beta\sum_{\bn \in J^\eps}\CH^d(E \cap Q_\bn^\eps)\\
      &\leq \tilde \beta\CH^d(E)\\
      &\le C\tilde \beta
      \CH^{d-1}(\del^*E)^{\frac{d}{d-1}}\\
      &=C\tilde \beta\left(\CH^{d-1}(\del^*E \cap \del \Omega) + \CH^{d-1}(\del^*E \cap
    \del T^\eps) + \CH^{d-1}(\del^*E \cap \Omega^\eps)\right)^{\frac{d}{d-1}}.
    \end{aligned}
  \end{equation}
  Together with \eqref{eq:extbdry} and \eqref{eq:locbdry} this leads to the existence of a constant $C$ which can depend on $\beta$, such that
    \begin{equation}
      \label{eq:2ndpart}
      \begin{aligned}
        \sum_{\bn \in J^\eps}\CH^{d-1}(\partial T_\bn^\eps)\le C \left(\CH^{d-1}(\del^*E \cap \Omega^\eps) + \sum_{\bn \in J^\eps}\CH^{d-1}(\del T_\bn^\eps)\right)^{\frac{d}{d-1}}.
  \end{aligned}
    \end{equation}
    Because $J^\eps\neq\emptyset$, dividing and factoring, this leads to
  \begin{align}
    1\le C\left(\sum_{\bn \in J^\eps}\CH^{d-1}(\partial T_\bn^\eps)\right)^{\frac{1}{d-1}}\left(\frac{\CH^{d-1}(\del^*E \cap \Omega^\eps)}{\sum_{\bn \in J^\eps}\CH^{d-1}(\partial T_\bn^\eps)} + 1\right)^{\frac{d}{d-1}}.
  \end{align}
  For each $\bn\in J^\eps$, $Q_\bn^\eps\subset B_{2\delta}(x_\eps)$ and it follows
  from \eqref{eq:choicedelta} that we can choose $\delta$ small
  enough depending on $\Omega$, the dimension, and $\beta$ but not on $\eps$
  such that
  \begin{equation}
     C\left(\sum_{\bn \in J^\eps}\CH^{d-1}(\partial
     T_\bn^\eps)\right)^{\frac{1}{d-1}}<\frac{1}{4}.
  \end{equation}
  This implies that there is a constant $c>0$ such that for all $\eps > 0$,
  \begin{equation}
    \label{eq:globbdry}
    1 \le \frac{\CH^{d-1}(\del^*E \cap \Omega^\eps)}{\sum_{\bn \in J^\eps} \CH^{d-1}(\del T_\bn^\eps)}.
  \end{equation}
  Combining \eqref{eq:extbdry2}, \eqref{eq:locbdry} and \eqref{eq:globbdry}, provides a constant $C$ such that
  \begin{equation}
    \CH^{d-1}(\del^*E \cap \del \Omega^\eps) \le C \CH^{d-1}(\del^*E \cap
    \Omega^\eps).
  \end{equation}
\end{proof}

\begin{remark}
  When $r_\eps = \smallo{\eps^{\frac{d}{d-1}}}$, (i.e. in the subcritical
  regime), the previous result along with \cite{bgt}[Theorem 4.1] implies
  convergence of the eigenvalues of the Steklov problem on $\Omega^\eps$ to the
  eigenvalues of the Steklov problem on $\Omega$. 
\end{remark}

\section{\bf Homogenisation of the Steklov problem} \label{sec:homogenisation}
 Let us first establish some basic facts related to the geometry of the
 homogenisation problem, under the assumption that $r_\eps^{d-1} \eps^{-d}
 \to \beta \in [0,\infty)$ as $\eps \to 0$. The number of holes $N(\eps)=\#I^\eps$ satisfies
$$N(\eps) \sim |\Omega|\eps^{-d}  \qquad\mbox{ as }\eps \to 0.$$
This implies,
\begin{gather}\label{eq:sizeholes}
  \abs{T^\eps} = \sum_{\bk\in I^\eps}|T_{\bk}^\eps| = \bigo{r_\eps}\qquad\mbox{and}\qquad
 \abs{\del T^\eps}=  \sum_{\bk\in I^\eps}|\partial T_{\bk}^\eps|\sim A_d \beta |\Omega|.
\end{gather}

The remainder of the section is split into three parts. In the first one we
extend the functions $u_k^\eps$ to the whole of $\Omega$, in order to obtain
weak $\RH^1$ convergence, up to taking subsequences. In the second part, we
prove that those converging subsequences converge to solutions of Problem
\eqref{problem:homo}. Finally, we prove in the third part that the only
functions they can converge to are the corresponding eigenfunction in
\eqref{problem:homo}, implying convergence as $\eps \to 0$, with this understood
in the sense of Remark \ref{rem:multiple} if the limit problem has eigenvalues
that are not simple.
\subsection{Extension of eigenfunctions} 
For $k \ge 1$, recall that $u_k^{\eps}:\Omega^\eps\rightarrow\R$ is the $k$'th Steklov
eigenfunction on $\Omega^\eps$:
\begin{gather*}
  \begin{cases}
    \Delta u_k^{\eps}=0&\mbox{ in }\Omega^\eps,\\
    \partial_{\nu}u_k^{\eps}=\sigma_k^{\eps}u_k^{\eps}&\mbox{ on
    }\partial\Omega^\eps.
  \end{cases}
\end{gather*}
Recall also that the eigenfunctions $u_k^{\eps}$ are normalized by requiring that
$$\int_{\partial\Omega^\eps}\left(u_k^{\eps}\right)^2 \de A=1.$$
Define the function $U_k^{\eps}\in \RH^1(\Omega)$ to be the harmonic extension of
$u_k^{\eps}$ to the interior of the holes:
\begin{gather*}
  \begin{cases}
    U_k^{\eps}=u_k^{\eps}&\mbox{ in }\overline{\Omega^\eps},\\
    \Delta U_k^{\eps}=0&\mbox{ in } T^\eps.
  \end{cases}
\end{gather*}

\begin{lemma}\label{lemma:boundedUn}
  There is a sequence $\eps_n \to 0$ such that $U_k^{\eps_n}$ has a weak
  limit in $\RH^1(\Omega)$.
\end{lemma}
\begin{proof}

  It suffices to show that $\set{U_k^\eps : 0 <\eps \le 1}$ is bounded in
  $\RH^1(\Omega)$.   Recall that 
  \begin{equation}
  \norm{U_k^{\eps}}^2_{H^1} = \norm{U_k^{\eps}}^2_{\RL^2} +
  \CD\left(U_k^{\eps}\right).
\end{equation}
We first bound the $\RL^2$ norm of $U_{k}^{\eps}$.
Let $\lambda$ be the first eigenvalue of the following Robin problem on $\Omega$:
\begin{gather*}
\begin{cases}
-\Delta u= \lambda u&\mbox{ in }\Omega,\\
\partial_\nu u= - u &\mbox{ on }\partial\Omega.
\end{cases}
\end{gather*}
It is well known (see \emph{e.g.} \cite{bucurfreitaskennedy}) that $\lambda>0$ and that it admits the following characterization:
$$\lambda=\inf_{v \in \RH^1(\Omega)} \frac{\int_\Omega|\nabla v|^2 \de
\bx+\int_{\partial\Omega}v^2\de S}{\int_{\Omega}v^2\de \bx}.$$
Applying this to $v=U_k^{\eps}$ leads to
\begin{equation}
  \label{eq:L2bound}
\begin{aligned}\int_\Omega \left(U_k^{\eps}\right)^2 \de \bx
  &\leq\frac{1}{\lambda}\left(\int_\Omega\left|\nabla U_k^{\eps}\right|^2 \de
  \bx+\int_{\partial\Omega}\left(U_k^{\eps}\right)^2 \de A\right)\\
  &\leq \frac{1}{\lambda}\left(\CD\left(U_k^{\eps}\right)+1\right).
  \end{aligned}
\end{equation}
  It is therefore sufficient to bound the Dirichlet energy. We first see that
\begin{equation}
  \label{eq:dirusep}
  \CD\left(U_k^{\eps};\Omega\right) = \CD\left(u_k^{\eps};\Omega^\eps\right) +
  \CD\left(U_k^{\eps};T^\eps\right)
  =\sigma_k^\eps+ \CD\left(U_k^{\eps};T^\eps\right).
\end{equation}
It follows from Lemma \ref{lemma:localEstimateFourier} and monotonicity of the Dirichlet energy,
that the contribution from the holes is
\begin{equation}
  \label{eq:dirutrous}
\begin{aligned}
\CD\left(U_k^{\eps};T^\eps\right)&=\sum_{\bk\in
I^\eps}\CD\left(U_{k}^{\eps};T_{\bk}^\eps\right)\\
&\leq
\sum_{\bk\in I^\eps} 5\left(\frac{r_\eps}{\eps}\right)^d
\CD\left(u_k^{\eps};Q_{\bk}^\eps\right)
\left( 1 + \bigo{\left(\frac{ r_\eps}{\eps} \right)^d} \right)
\\
&\leq 
5 \left(\frac{r_\eps}{\eps}\right)^d \CD\left(u_k^{\eps};\Omega^\eps\right) 
\left( 1 + \bigo{\left( \frac{r_\eps}{\eps} \right)^d} \right) \\
     &\leq C \sigma_k(\Omega^\eps),
\end{aligned}
\end{equation}
for some constant $C$.  Combining \eqref{eq:dirutrous} and
  \eqref{eq:L2bound}, we see that, to bound $\norm{U_k^\eps}_{\RH^1(\Omega)}$,
  it is sufficient to find a bound for 
  $\sigma_k^\eps$ independent of $\eps$. The variational characterisation for
  Steklov eigenvalues can be rewritten as
  \begin{equation}
    \sigma_k^\eps = \min_{\substack{E \subset \RL^2(\del \Omega^\eps) \\ \dim(E)
    = k+1}} \max_{u \in E} \frac{\CD(u)}{\norm{u}^2_{\RL^2(\del\Omega^\eps)}}.
  \end{equation}
  We use eigenfunctions of the dynamical eigenvalue problem as a test subspace
  for $\sigma_k^\eps$. Namely, setting $E = \operatorname{span}(U_0,\dotsc,U_{k})$
  we see that for $\eps$ small enough it spans a $k+1$ dimensional subspace of
  $\RL^2(\del\Omega^\eps)$. Indeed, they are an orthonormal set with respect to
  $(\cdot,\cdot)_\beta$, and for every $0 \le j,\ell \le k$,
  \begin{equation}
    (U_j,U_\ell)_{\del^\eps} \xrightarrow{\eps \to 0} (U_j,U_\ell)_{\beta}.
  \end{equation}
  Therefore, from the characterisation of the eigenvalues $\Sigma_{k,\beta}$,
  \begin{equation}
    \begin{aligned}
    \sigma_k^\eps &\le \max_{U \in E} \frac{\CD(U)}{(U,U)_{\del^\eps}} \\
    &\le \Sigma_{k,\beta} + \smallo{1}.
  \end{aligned}
  \end{equation}
  This completes the proof that $\set{U_k^\eps}$ is bounded so that there is
  a converging subsequence as $\eps \to 0$.

\end{proof}

From now on we will abuse notation and relabel that sequence $\eps_n \to 0$
along which $U_k^\eps$ has a weak limit as $\eps \to 0$ again.

\subsection{Establishing the limit problem}
Our aim by the end of this subsection is to prove the following weaker version
of Theorem \ref{thm:homogenisation}.
\begin{prop}\label{prop:weak}
  Let $k\in\N$.
  As $\eps\to 0$, the pairs $\left(\sigma_k^\eps,U_k^\eps\right)$ converge to a solution
  $(\Sigma,U)$ of \eqref{problem:homo}, the convergence of the functions $U_k^\eps$
  being weak in $\RH^1(\Omega)$. 
\end{prop}

Up to choosing a subsequence, we assume that $\sigma_k^{\eps}$ converges to some
number $\Sigma$
and also that
$\set{U_k^{\eps}}\subset \RH^1(\Omega)$ is weakly converging in $\RH^1(\Omega)$
to some $U \in \RH^1(\Omega)$, from which we
also get strong convergence to $U$ in $\RL^2(\Omega)$.
Considering the real-valued test function $V\in\RH^1(\Omega)$, we see that
\begin{align*}
  \int_\Omega \nabla U_k^{\eps}\cdot\nabla V \de \bx
&=
\int_{\Omega^\eps} \nabla u_k^{\eps}\cdot\nabla V \de \bx+\int_{T^\eps} \nabla
U_k^{\eps}\cdot\nabla V \de \bx\\
&=\sigma_k^{\eps}\int_{\partial\Omega^\eps}u_k^{\eps}V \de \bx+\int_{T^\eps} \nabla
U_k^{\eps}\cdot\nabla V \de \bx\\
&=\sigma_k^{\eps}\int_{\partial\Omega}u_k^{\eps}V \de
A+\sigma_k^{\eps}\int_{\partial T^\eps}u_k^{\eps}V \de A +\int_{T^\eps} \nabla
U_k^{\eps}\cdot\nabla V \de \bx.
\end{align*}
Letting $\eps\to0$ leads, if the limits exist, to
\begin{gather*}
\int_\Omega \nabla U\cdot\nabla V \de \bx-\Sigma\int_{\partial\Omega}UV
\de A=
\Sigma\lim_{\eps \to 0}
\int_{\partial T^\eps}u_k^{\eps} V \de A+\lim_{\eps \to
0}\int_{T^\eps} \nabla
  U_k^{\eps}\cdot\nabla V \de\bx.
\end{gather*}
It follows from the Cauchy--Schwarz inequality that
$$\int_{T^\eps} \nabla U_k^{\eps}\cdot\nabla V \de \bx\leq\left(\int_{T^\eps} \left|\nabla
U_k^{\eps}\right|^2 \de \bx\int_{T^\eps}|\nabla V|^2 \de \bx\right)^{1/2},$$
which tends to 0 according to Lemma \ref{lemma:localEstimateFourier}.
It follows that
\begin{gather*}
\int_\Omega \nabla U\cdot\nabla V \de \bx-\Sigma\int_{\partial\Omega}UV
\de A=
\Sigma\lim_{\eps \to 0}\int_{\partial T^\eps}u_k^{\eps}V \de A,
\end{gather*}
and all that is left to do is to analyse the last term.

\begin{prop}\label{prop:limitproblem}
  Suppose that $\eps^{-d}r_\eps^{d-1}\to\beta\geq 0$. Then,
	for each $V\in \RH^1(\Omega)$ the following holds,
    \begin{equation}
      \label{eq:toprove}
      \lim_{\eps \to 0}\int_{\partial T^\eps}u_k^{\eps}V \de A=
      A_d\beta\int_{\Omega}UV \de \bx.
    \end{equation}
\end{prop}

\begin{remark}
    The functional $V\mapsto\int_{\partial T^\eps}u_k^{\eps}V$ is bounded on $H^1(\Omega)$. By the 
    Riesz--Fr\'echet representation theorem, there exists
    a function $\xi^\eps\in H^1(\Omega)$ such that
    $$\int_{\partial T^\eps}u_k^{\eps}V =\int_{\Omega}\nabla\xi^\eps\cdot
    \nabla V+\xi^\eps V\,dx\qquad\forall V\in H^1(\Omega).$$
    Using appropriate test functions shows that $\xi^\eps$ is the weak solution of the following transmission problem:
    \begin{gather}\label{Problem:transmission}
      \begin{cases}
        \Delta\xi^\eps=0&\mbox{ in }\Omega^\eps\cup T^\eps,\\
        \partial_\nu \xi^\eps_++\partial_\nu \xi^\eps_-=u_k^\eps&\mbox{ on }\partial T^\eps,\\
        \partial_\nu \xi^\eps=0&\mbox{ on }\partial\Omega.
      \end{cases}
    \end{gather}
    Proposition~\ref{prop:limitproblem} is an homogenisation result for this problem. It means that in the limit as $\eps\to 0$, the solution converges to that of the following problem:
    \begin{gather}
      \begin{cases}
        -\Delta \Xi+(1-A_d\beta)\Xi=0&\mbox{ in }\Omega,\\
        \partial_\nu \Xi=0&\mbox{ on }\partial\Omega.
      \end{cases}
    \end{gather}
    Transmission problems have recently been the subject of investigation through means of homogenisation, see for example \cite{dmfz}.
  \end{remark}

   The proof of Proposition \ref{prop:limitproblem} is divided in three main
  steps. In the first step, we justify that we can use smooth test functions in
the limit \eqref{eq:toprove}. In order to do so, we will use an inner
representation of the lefthandside in \eqref{eq:toprove}, defined in terms of
extensions to the holes, to show that it is bounded, uniformly in
$\eps$. In this representation, however, it is hard to explicitly compute the
limit problem. 

In our second step, we introduce an outer representation of the
lefthandside in \eqref{eq:toprove}, through integration
on the outside of the holes. This representation is given in terms of an
auxilliary function $\Psi$, for which we derive some regularity properties.

In the final step, we use this latter representation to show that the limit
\eqref{eq:toprove} indeed holds. Here, we reap rewards from the previous steps
and use explicitly the properties of the
auxilliary function $\Psi$, as well as better estimates awarded from the fact
that we can test against smooth functions.

\begin{proof}[Proof of Proposition \ref{prop:limitproblem}]

  Define the family of bounded functionals $L_\eps:\RH^1(\Omega)\rightarrow\R$ by
  $$L_\eps(V):=\int_{\partial T_\eps} V \de A.$$
  
	\noindent\textbf{Step 1}: Inner representation of $L_\eps$.
        
        Define $\phi_\eps:\R^d\rightarrow\R$ by
        \begin{equation}\label{eq:explicitphi}
          \phi_\eps(\bx) =
          \begin{cases}
            \frac{ \abs \bx^2}{2r_\eps}&\text{for }x\in B(0,r_\eps),\\
            0&\text{elsewhere}.
          \end{cases}
        \end{equation}
        By periodizing along $\eps\Z^d$ we obtain the function
        $\Phi_\eps : \R^d \to \R$ given by
	\begin{equation}
	  \Phi_\eps(\bx) := \sum_{\bk \in I^\eps} \phi_\eps(\bx - \eps \bk).
	\end{equation}
        \begin{lemma}\label{lemma:innerrep}
          The functional $L_\eps:\RH^1(\Omega)\rightarrow\R$
          admits the following representation:
          \begin{equation}\label{eq:innerrep}
            L_\eps(V) = \frac{d}{r_\eps} \int_{T^\eps} V \de \bx + \int_{T^\eps} \nabla \Phi_\eps \cdot \nabla V \de \bx.
          \end{equation}
        \end{lemma}
        \begin{proof}
          It is straightforward to check that
          \begin{equation}
	    \begin{cases}
	      \Delta \Phi_\eps = \frac{d}{r_\eps} & \text{in } T^\eps,\\
	      \del_\nu \Phi_\eps = 1 & \text{on } \del T^\eps,\\
	      \Phi = 0 & \text{in } \Omega^\eps.
	    \end{cases}
	  \end{equation}
          The function $\Phi_\eps$ therefore satisfies the weak identity
          \begin{equation}
            \int_{T^\eps} \nabla \Phi_\eps \cdot V = - \frac{d}{r_\eps} \int_{T^\eps} V + \int_{\del T^\eps} V,\qquad\forall V\in H^1(\Omega).
          \end{equation}
        \end{proof}
    
For each $k\in\N$ and $\eps>0$, the functional $\tilde{L_\eps}:H^1(\Omega)\rightarrow\R$ is defined by
$$\tilde{L_\eps}(V):=L_\eps(u_k^{\eps}V).$$
\begin{lemma} \label{lem:banstein}
There is an $\eps_0>0$ such that the family
$\left\{\tilde{L_\eps}\right\}_{\eps>0}\subset \left(H^1(\Omega)\right)^*$
is uniformly bounded for $0 < \eps \le \eps_0$.
\end{lemma}
\begin{proof}
Given $V \in \RH^1(\Omega)$,
  it follows from the Lemma~\ref{lemma:innerrep} that
  \begin{equation}
	\label{eq:innerrepbound}
	\begin{aligned}
		L_\eps(U_k^\eps V) = \frac{d}{r_\eps} \int_{T^\eps} U_k^\eps V \de \bx + \int_{T^\eps} \nabla \Phi_\eps \cdot \nabla(U_k^\eps V) \de \bx.
	\end{aligned}
\end{equation}
To bound the first term, start by using the Cauchy--Schwarz inequality:
\begin{equation}\label{eq:innerrepbound2}
	\abs{\frac{d}{r_\eps}\int_{T^\eps} U_k^\eps V \de \bx} \le \frac{d}{r_\eps}\norm{U_k^\eps}_{\RL^2(T^\eps)} \norm{V}_{\RL^2(T^\eps)}.
\end{equation}
It follows from Lemma~\ref{lemma:smalltolarge} that there is $C > 0$ depending
only on $\beta$ such that
  \begin{equation}\label{ineq:demiborne1}
    \begin{aligned}
      \norm{V}_{\RL^2(T^\eps)} &\le C r_\eps^{1/2} \norm{V}_{\RH^1(\Omega)},\\
      \norm{U_{k}^\eps}_{\RL^2(T^\eps)} &\le C r_\eps^{1/2}
      \norm{U_k^\eps}_{\RH^1(\Omega)},
  \end{aligned}
  \end{equation}
  so that
  \begin{equation}
    \label{ineq:demiborne2}
    \sup_{\eps \in (0,1)} \abs{\frac{d}{r_\eps} \int_{T^\eps}U_k^\eps V \de \bx
    } < \infty.
  \end{equation}

To bound the second term in~\eqref{eq:innerrepbound}, the generalised Hölder inequality leads to
\begin{equation*}
  \begin{aligned}
    \abs{\int_{T^\eps} \nabla \Phi_\eps \cdot \nabla(U_k^\eps V) \de \bx}
    &\le \abs{\int_{T^\eps} U_k^\eps \nabla \Phi_\eps \cdot \nabla V + V \nabla \Phi_\eps \cdot \nabla U_k^\eps \de \bx} \\
    &\le \norm{U_k^\eps}_{\RL^2(T^\eps)} \norm{\nabla V}_{\RL^2(T^\eps)}\norm{\nabla \Phi_\eps}_{\RL^\infty(T^\eps)}\\
    &\qquad+ \norm{\nabla U_k^\eps}_{\RL^2(T^\eps)} \norm{V}_{\RL^2(T^\eps)} \norm{ \nabla \Phi_\eps}_{\RL^\infty(T^\eps)}\\
    &\leq\left(\norm{U_k^\eps}_{\RL^2(T^\eps)}+\norm{\nabla U_k^\eps}_{\RL^2(T^\eps)}\right)\norm{\nabla V}_{H^1(\Omega)}\norm{\nabla \Phi_\eps}_{\RL^\infty(T^\eps)}\\
    &\leq\norm{U_k^\eps}_{H^1(T^\eps)}\norm{\nabla V}_{H^1(\Omega)}.
  \end{aligned}
\end{equation*}
In the last inequality we have used $\norm{\nabla \Phi_\eps}_{\RL^\infty(T^\eps)}=1$, which follows from~\eqref{eq:explicitphi}.
This quantity is uniformly bounded as $\epsilon\searrow 0$ since we have shown in the proof of Lemma~\ref{lemma:boundedUn} that  $U_k^\eps$ is bounded in $H^1(\Omega)$.
Together with~\eqref{ineq:demiborne2} this proves for each $V\in H^1(\Omega)$ the existence of a constant $C$ such that
$|L_\eps(U_k^\eps V)|\leq C\|V\|$ for each $\eps$,
and the conclusion follows from the Banach--Steinhaus theorem.
\end{proof}

\medskip
\noindent\textbf{Step 2}: Outer representation of $L_\eps$.
Consider the torus $\CC=\T^d = \R^d / \Z^d$ and introduce the fundamental cell $\CC^\eps$ as the perforated torus
$$\CC^\eps:=\CC\setminus \overline{B\left(0,\rho_\eps\right)},$$
where $\rho_\eps:= \eps^{-1} r_\eps$ is the renormalised radius.
Following \cite{vanninathan}, we define the function 
$\psi_\eps\in \RH^1(\CC^\eps)$ through the weak
variational problem:
\begin{equation}
  \label{eq:psivar}
  \int_{\CC^\eps}\nabla\psi_\eps\cdot\nabla
  V=-c_\eps\int_{\CC^\eps}V+\int_{\partial B(0,\rho_\eps)}V.
\end{equation}
By taking $V\equiv 1$, one sees that the necessary and sufficient condition for existence
of a solution (see e.g. \cite[Theorem 5.7.7]{TaylorI}) is 
$$c_\eps=\frac{A_d \rho_\eps^{d-1}}{|\CC^\eps|}\sim A_d (\rho_\eps)^{d-1}.$$
Uniqueness of the solution is guaranteed by requiring that $\psi_\eps$
be orthogonal to constants on $\CC^\eps$. Therefore, $\psi_\eps$ is the
unique function such that
\begin{gather*}
  \begin{cases}
    \Delta\psi_\eps=c_\eps&\mbox{ in }\CC^\eps\\
    \partial_\nu\psi_\eps=1&\mbox{ on }\partial B(0,\rho_\eps)
  \end{cases}
  \qquad\mbox{ and }\qquad
  \int_{\CC^\eps}\psi_\eps=0.
\end{gather*}

Consider the union of all cells strictly contained in $\Omega$,
$$\tilde\Omega^\eps:=\bigcup_{\bk\in I^\eps}Q_{\bk}^\eps\subset\Omega^\eps.$$
Define the function $\Psi_\eps: \R^d \setminus \bigcup_{\bk \in \Z^d}
T_{\bk}^\eps \rightarrow\R$
as the scaled lift of $\psi_\eps$. That is, if $q : \R^d \to \R^d/\Z^d$ is the
covering map, then 
$$\Psi_\eps(\eps \bx):= \psi_\eps(q(\bx)).$$
This function satisfies
\begin{gather*}
\begin{cases}
\Delta\Psi_\eps = \eps^{-2} c_\eps &\mbox{ in }\tilde\Omega^\eps,\\
\partial_\nu\Psi_\eps= \eps^{-1} &\mbox{ on }\partial T^\eps.
\end{cases}
\end{gather*}
\begin{lemma}The functional  $L_\eps:\RH^1(\Omega)\rightarrow\R$ 
admits the following representation:
$$L_\eps(V)= \eps\int_{\tilde\Omega^\eps}\nabla\Psi_\eps\cdot\nabla V \de \bx
+ \eps \int_{\del \tilde \Omega^\eps \setminus \del T^\eps} V \del_\nu \Psi_\eps
\de A
+\eps^{-1}c_\eps\int_{\tilde\Omega^\eps}V \de A.$$
\end{lemma}

The proof is immediate since for $V\in \RH^1(\Omega)$ the following holds:
\begin{equation}
\int_{\tilde\Omega^\eps}\nabla\Psi_\eps\cdot\nabla V \de \bx=
-\eps^{-2}c_\eps\int_{\tilde\Omega^\eps}V \de \bx+\eps^{-1}\int_{\partial
T^\eps}V \de A
+ \int_{\del \tilde \Omega^\eps \setminus \del T^\eps} V \del_\nu \Psi_\eps \de
A.
\end{equation}

We establish the following claim concerning $\psi_\eps$. 
\begin{lemma}
  \label{lemma:sobolevpsieps} 
  There is a constant $C$, depending only on the dimension and on $\beta$, such that
  \begin{equation}
    \norm{\psi_\eps}_{\RH^1(\CC^\eps)} \le C \eps^{\frac 1 2 + \frac 1 d}.
  \end{equation}
  Furthermore, for any $s > 1$, any compact set $K \subset \CC$, containing the origin in its
  interior, there is a constant $C'$ depending only on the dimension,
  $\beta$, $s$,
  and on $K$ such that
  \begin{equation}
    \label{eq:hspsi}
      \norm{\psi_\eps}_{\RH^s(\CC^\eps \setminus K)} \le  C' \eps^{\frac 1
        2 + \frac 1 d}.
  \end{equation}
  In particular, this implies $\norm{D^\alpha \psi_\eps}_{\RL^\infty(\CC^\eps \setminus
    K)}$ decays
  as $\eps^{\frac 1 2 + \frac 1 d}$ for any multi-index $\alpha$. 
\end{lemma}
\begin{proof}
  Observe that since $\psi_\eps$ has mean $0$ on $\CC^\eps$, 
  the Poincaré--Wirtinger inequality implies
  \begin{equation}
    \label{eq:pw}
    \mu_1(\CC^\eps) \norm{\psi_\eps}^2_{\RL^2(\CC^\eps)} \le \norm{\nabla
      \psi_\eps}^2_{\RL^2(\CC^\eps)^d},
  \end{equation}
  where $\mu_1(\CC^\eps)$ is the first non-zero Neumann eigenvalue of $\CC^\eps$. Observe that
  $$\mu_1(\CC^\eps) \to 4\pi^2\qquad\text{as }\eps \to 0.$$
  Indeed, $\mu_1(\CC^\eps)$ 
  is the first non-zero Neumann eigenvalue of a punctured $d$-dimensional
  torus, which is known to converge to the first nonzero eigenvalue of the torus
  itself as $\eps\searrow 0$. See~\cite[Chapter IX]{chaveleig} for instance.
  
  Take $V = \psi_\eps$ in the variational characterisation
  \eqref{eq:psivar} of $\psi_\eps$, and consider $\eps$ to be small enough
  that $\mu_1(\CC^\eps) \ge 1$. Using the Cauchy--Schwarz inequality
  yields
  \begin{equation}
    \label{eq:gradpsi}
    \begin{aligned}
      \norm{\psi_\eps}_{\RH^1(\CC^\eps)}^2 &\le 2\int_{\CC^\eps}|\nabla\psi_\eps|^2\de \bx  \\
      &\le 2\int_{\partial
        B(0,\rho_\eps)}\psi_\eps \de A\\
      &\leq 2\sqrt{A_d \rho_\eps^{d-1}}\|\psi_\eps\|_{\RL^2(\partial
        B(0,\rho_\eps))}\\
      &\leq 2\sqrt{A_d\rho_\eps^{d-1}}\|\tau_\eps\|\|\psi_\eps\|_{\RH^1(\CC^\eps)},
    \end{aligned}
  \end{equation}
  where $\tau_\eps$ is the trace operator $\RH^1(\CC^\eps) \to \RL^2(\del
  B(0,\rho_\eps))$. From the definition of $\tau_\eps$ and monotonicity of the
  involved integrals, we have that
\begin{equation}
\norm{\tau_\eps} \le \gamma\left(\eps^{\frac{1}{d-1}},1\right)^{-1},
\end{equation}
where $\gamma$ is defined in Lemma \ref{lemma:Sobolevconst}. We therefore deduce
that $\norm{\tau_\eps} \ll \eps^{1/d}$, the implicit constant depending only on
the dimension and on $\beta$. Finally, dividing both sides in \eqref{eq:gradpsi}
by $\norm{\psi_\eps}_{\RH^1}$, observing that $\rho_\eps^{d-1} \sim \beta\eps$ and
inserting the bound for $\norm{\tau_\eps}$ in \eqref{eq:gradpsi} finishes the
proof of the first inequality.

For the second one, we have from \cite[Theorem 8.10]{gilbargtrudinger} that there is
a constant $\tilde C$ depending only on the dimension, on $K$, and on $s$ such that
\begin{equation}
  \norm{\psi_\eps}_{\RH^s(\CC^\eps \setminus K)} \le \tilde C \left(
  \norm{\psi_\eps}_{\RH^1(\CC^\eps)} + \norm{c_\eps}_{\RH^s(\CC^\eps)}
\right). 
\end{equation}
The second term is of order $\rho_\eps^{d-1} \sim \eps$, and bounds for the
$\RH^1$ norm of $\psi_\eps$ that were obtained in \eqref{eq:gradpsi} conclude
the proof of the second inequality.

As for the remark considering the $\RL^\infty$ bounds on derivatives of
$\psi_\eps$, we have from inequality \eqref{eq:hspsi} that
\begin{equation}
  \norm{D^\alpha \psi_\eps}_{\RH^{s-\abs \alpha}(\CC^\eps \setminus K)} 
  \le \norm{\psi_\eps}_{\RH^s(\CC^\eps \setminus K)} \le
  C'\eps^{\frac 1 2 + \frac 1 d}.
\end{equation}
Choosing $s > \abs \alpha + \frac{d+1}{2}$, and $\eps$ small enough that
$\CC^\eps \setminus K = \CC \setminus K$, we have that
$$\RH^{s- \abs \alpha}(\CC^\eps\setminus K) \hookrightarrow \RL^\infty(\CC^\eps
\setminus K)$$
with norm  independent of $\eps$, concluding the proof.
    \end{proof}

\noindent
\textbf{Step 3}: Computing the limit problem.
It follows from Lemma \ref{lem:banstein} that it is sufficient to verify convergence for a test function $V \in C^\infty(\overline\Omega)$. From the outer representation for $L_\eps$ we have that
\begin{equation}
  \label{eq:outerreptilde}
  \begin{aligned}
    \tilde L_\eps(V) &= \int_{\partial T^\eps}u_k^{\eps}V \\
    &= \eps \int_{\tilde\Omega^\eps}\nabla\Psi_\eps\cdot\nabla (u_k^{\eps}V) +
    \int_{\del \tilde \Omega^\eps \setminus \del T^\eps} u_k^\eps V \del_\nu
    \Psi_\eps
    +\underbrace{\eps^{-1}c_\eps\int_{\tilde\Omega^\eps}u_k^{\eps}V}_{\to
    A_d\beta\int_\Omega UV},
\end{aligned}
\end{equation}
where convergence of the last term stems from strong $\RL^2$ convergence.
We now show that the two other terms converge to $0$. For the first one, 
the Cauchy--Schwarz inequality gives
\begin{equation}
  \label{eq:afterCS}
  \int_{\tilde\Omega^\eps} \eps \nabla\Psi_\eps\cdot\nabla (u_k^{\eps}V)\leq
  \left(\int_{\tilde\Omega^\eps} \eps^{2}|\nabla\Psi_\eps|^2
  \int_{\tilde\Omega^\eps}|\nabla (u_k^{\eps}V)|^2\right)^{1/2}.
\end{equation}
Let us first observe that since $V$ is in $ C^\infty(\overline \Omega)$, we have that 
\begin{equation}
  \begin{aligned}
    \int_{\tilde \Omega^\eps} \abs{\nabla(u_k^{\eps} V)}^2 
    &\le 2 \sup_{\bx \in \tilde \Omega^\eps} \left( \abs{V(\bx)}^2 + \abs{\nabla V(\bx)}^2
    \right) \|U_k^{\eps}\|_{\RH_1(\Omega)} \\
    &= \bigo 1.
  \end{aligned}
  \label{eq:boundgraduv}
\end{equation}
On to the other term, note that
\begin{equation}
  \label{eq:Psitopsi}
\begin{aligned}
\int_{\tilde\Omega^\eps} \eps^2 |\nabla\Psi_\eps|^2&=
\eps^2 \sum_{\bk\in I^\eps}\int_{Q^\eps_{\bk}} |\nabla\Psi_\eps|^2 \\
&=\sum_{\bk\in I^\eps}\eps^{d}\int_{\CC^\eps}|\nabla\psi_\eps|^2 \\
&\sim |\Omega|\int_{\CC^\eps}|\nabla\psi_\eps|^2.
\end{aligned}
\end{equation}
It follows from Lemma \ref{lemma:sobolevpsieps} that 
\begin{equation}
  \int_{\CC^\eps} \abs{\nabla \psi_\eps}^2 \de \bx = \bigo{\eps^{1 + \frac 2 d}},
\end{equation}
hence that term indeed goes to $0$. For the second term in
\eqref{eq:outerreptilde}, we have from the generalised Hölder inequality that
\begin{equation}
  \eps \int_{\del \tilde \Omega^\eps \setminus \del T^\eps} u_k^\eps V \del_\nu
  \Psi_\eps \de A \le \eps \norm{V}_{\RL^2(\del \tilde \Omega^\eps \setminus \del T^\eps)}
  \norm{u_k^\eps}_{\RL^2(\del \tilde \Omega^\eps \setminus \del T^\eps)}
\norm{\del_\nu \Psi_\eps}_{\RL^\infty(\del \tilde \Omega^\eps \setminus \del
T^\eps)},
\end{equation}
we now analyse each of those norms. 

First, by scaling of derivatives we have
\begin{equation}
\begin{aligned}
 \norm{\eps \del_\nu \Psi_\eps}_{\RL^\infty(\del \Omega^\eps \setminus \del T^\eps)} &\le 
 \sup_{\abs \alpha = 1}\norm{D^\alpha \psi_\eps}_{\RL^\infty(\CC^\eps \setminus B(0,1/4))} \\
 &\ll \eps^{\frac 1 2 + \frac 1 d},
 \end{aligned}
\end{equation}
where the last bound holds by Lemma \ref{lemma:sobolevpsieps}. We move on to
$\norm{u_k^\eps}_{\RL^2(\del \tilde \Omega^\eps \setminus \del T^\eps)}$. Denote
  by $\tilde I^\eps$ the set of indices $\bn \in \Z^d $ such that $\del
  Q_\bn^\eps \cap \del \tilde \Omega^\eps \ne \varnothing$. One can see that
\begin{equation}
  \label{eq:lastestimate}
  \norm{u_k^\eps}_{\RL^2(\del \tilde \Omega^\eps \setminus \del T^\eps)} \le \sum_{\bn \in \tilde I^\eps}\norm{U_k^\eps}_{\RL^2(\del Q_\bn^\eps)}.
\end{equation}
On the other hand, it follows from Lemma \ref{lem:traceconstant} that there is a constant $C$ (in fact, the trace constant of the unit cube) such that
\begin{equation}
\begin{aligned}
 \sum_{\bn \in \tilde I^\eps}\norm{U_k^\eps}_{\RL^2(\del Q_\bn^\eps)} &\le C\eps^{-1/2} \sum_{\bn \in \tilde I^\eps}\norm{U_k^\eps}_{\RH^1(Q_{\bn}^\eps)} \\
 &\le C\eps^{-1/2} \norm{U_k^\eps}_{\RH^1(\Omega)}.
 \end{aligned}
\end{equation}
Finally, since $V$ is fixed and smooth we have that
\begin{equation}
\begin{aligned}
 \norm{V}_{\RL^2(\del \tilde \Omega^\eps \setminus \del T^\eps)} &\le \abs{\del \tilde \Omega^\eps \setminus \del T^\eps}^{1/2} \sup_{\bx} V(\bx) \\
 &\le C \abs{\del \Omega}^{1/2} \sup_\bx V(\bx).
 \end{aligned}
\end{equation}
All in all, this implies that the second term in \eqref{eq:outerreptilde} is bounded by a constant times $\eps^{1/d}$, hence goes to $0$
as $\eps \to 0$, finishing the proof of Lemma \ref{prop:limitproblem}.
\end{proof}

We are now ready to complete the proof of the main result of this subsection.
\begin{proof}[Proof of Proposition~\ref{prop:weak}]
  We know from \cite{BelowFrancois} that we only need to show that the solutions
converge to a solution $(\Sigma,U)$ of the weak formulation of Problem \ref{problem:homo},1
  which is that for any test function $V$,
  \begin{equation}
    \int_\Omega \nabla U \cdot \nabla V \de \bx = \Sigma\left(A_d \beta
    \int_\Omega U V \de \bx + \int_{\del \Omega} U V \de A\right).
  \end{equation}
  The weak formulation of the Steklov problem on $\Omega^\eps$ is that for all test functions $V$,
  \begin{equation}
    \int_{\Omega^\eps} \nabla u_k^\eps \cdot \nabla V \de \bx = \sigma_k^\eps
    \left(\int_{\del \Omega} u_k^\eps V \de A + \int_{\del T^\eps} u_k^\eps V
    \de A\right).
  \end{equation}
  The convergence of the gradient terms follows from weak convergence in $\RH^1$
  of $U_k^\eps$ to $U$ and Lemma \ref{lemma:localEstimateFourier}. We already
  have that $\sigma_k^\eps \to \Sigma$. The integrals on $\del \Omega$
  converge by weak convergence in $\RH^1$ and compactness of the trace operator
  on $\del \Omega$. Finally, convergence of the interior term comes from
  Proposition
  \ref{prop:limitproblem}.
\end{proof}

\subsection{Spectral convergence of the problem}

We need the following technical lemma.
\begin{lemma} \label{lem:normconv}
 As $\eps \to 0$, we have that
  \begin{equation}
    \tilde L_\eps(u_k^\eps) \to A_d \beta \int_{\Omega} U^2 \de \bx.
  \end{equation}
\end{lemma}  
\begin{remark}
  Observe that this situation is specific to the sequence $u_k^\eps$.
  Indeed, there are sequences $\set{v_\eps}$ converging
  weakly to some $v \in \RH^1(\Omega)$ such that
  \begin{equation}
  \lim_{\eps \to 0} \tilde L_\eps (v_\eps) \ne A_d \beta \int_\Omega U v \de
  \bx.
  \end{equation}
\end{remark}
\begin{proof}

  From the outer representation \eqref{eq:outerreptilde}, we have that
\begin{equation}
  \begin{aligned}
    \tilde L_\eps(u_k^\eps) 
    &= \eps \int_{\tilde\Omega^\eps}\nabla\Psi_\eps\cdot\nabla (u_k^{\eps})^2
    \de \bx +
\eps \int_{\del \tilde \Omega^\eps \setminus \del T^\eps} (u_k^\eps)^2 \del_\nu
\Psi_\eps \de \bx +
    \eps^{-1}c_\eps\int_{\tilde\Omega^\eps}(u_k^{\eps})^2 \de
  \bx, \\
    &= \eps \int_{\tilde\Omega^\eps}2\nabla\Psi_\eps\cdot u_k^\eps\nabla
    u_k^\eps \de \bx
+\eps \int_{\del \tilde \Omega^\eps \setminus \del T^\eps} (u_k^\eps)^2 \del_\nu \Psi_\eps \de \bx
    +\eps^{-1}c_\eps\int_{\tilde\Omega^\eps}(u_k^{\eps})^2 \de
  \bx. \\
  \end{aligned}
\end{equation}
The last term converges towards the desired $A_d \beta \int_\Omega U^2 \de \bx$
once again by strong $L^2$ convergence of the sequence $U_k^\eps$. 
To study the first term, let us now introduce the sets
\begin{equation}
 \omega^\eps := \bigcup_{\bk \in I^\eps} B\left(\eps \bk,\frac \eps 4\right) \setminus B(\eps\bk,r_\eps) \subset \tilde \Omega^\eps,
\end{equation}
and decompose
\begin{equation}
  \eps \int_{\tilde \Omega^\eps} \nabla \Psi_\eps \cdot \nabla(u_k^\eps)^2
  \de \bx =
  2 \left(\int_{\omega^\eps} + \int_{\tilde \Omega^\eps \setminus \omega^\eps} \right) \eps \nabla \Psi_\eps \cdot u_k^\eps \nabla u_k^\eps \de
  \bx.
\end{equation}
 Let us first consider the integral over $\omega^\eps$. It follows from \cite{bgt} that the
$\RL^\infty$ norms of the Steklov eigenfunctions $u_k^\eps$ is bounded, uniformly in terms of
$\sigma_k^\eps$, and the norm of the trace $T : \RB\RV(\Omega^\eps) \to
\RL^1(\del \Omega^\eps)$, which we have shown in Lemma \ref{lem:tracebv} to be
bounded. This, along with the
scaled $\RH^1$ norm estimate for $\eps \nabla \Psi_\eps$ from Lemma
\ref{lemma:sobolevpsieps}, the $\RL^2$ boundedness of $\nabla u_k^\eps$ from
Lemma \ref{lemma:boundedUn} and the generalised Hölder inequality yields

\begin{equation}
   \label{eq:gpsigu2}
   \begin{aligned}
    \int_{\omega^\eps}2 \eps \nabla \Psi_\eps \cdot u_k^\eps \nabla
    u_k^\eps \de \bx &\le 2 \norm{\eps \nabla \Psi_\eps}_{\RL^2(\Omega^\eps)^d}
    \norm{\nabla u_k^\eps}_{\RL^2(\Omega^\eps)^d}
    \norm{u_k^\eps}_{\RL^\infty(\omega^\eps)} \\
    &\ll \eps^{\frac 1 2 + \frac 1 d}.
    \end{aligned}
\end{equation}

For the integral over $\tilde \Omega^\eps \setminus \omega^\eps$, observe that scaling Lemma \ref{lemma:sobolevpsieps} yields
\begin{equation}
\begin{aligned}
 \norm{\eps\nabla \Psi_\eps}_{\RL^\infty(\tilde \Omega^\eps \setminus \omega^\eps)^d} &= \sup_{\abs \alpha = 1} \norm{D^\alpha \psi_\eps}_{\RL^\infty(\CC^\eps \setminus B(0,1/4))} \\
 &\ll \eps^{\frac 1 2 + \frac 1 d}.
 \end{aligned}
\end{equation}

Inserting that bound into the generalised Hölder inequality, along with $\RH^1$ boundedness of $u_k^\eps$ yields
\begin{equation}
   \begin{aligned}
    \int_{\omega^\eps}2 \eps \nabla \Psi_\eps \cdot u_k^\eps \nabla
    u_k^\eps \de \bx &\le 2 \norm{\eps \nabla \Psi_\eps}_{\RL^\infty(\tilde \Omega^\eps \setminus \omega^\eps)^d}
    \norm{\nabla u_k^\eps}_{\RL^2(\Omega^\eps)^d}
    \norm{u_k^\eps}_{\RL^2(\Omega^\eps)} \\
    &\ll \eps^{\frac 1 2 + \frac 1 d}.
    \end{aligned}
\end{equation}

Finally, for the integral on the boundary $\del \tilde \Omega^\eps \setminus \del T^\eps$, we have
from Hölder that
\begin{equation}
  \label{eq:delnupsiuk2}
\eps \int_{\del\tilde \Omega^\eps \setminus \del T^\eps} (u_k^\eps)^2 \del_\nu
\Psi_\eps \de A \le \eps\abs{\del \tilde \Omega^\eps \setminus \del T^\eps}\norm{\del_\nu \Psi_\eps}_{\RL^\infty(\del \tilde \Omega^\eps \setminus \del
T^\eps)}\norm{(u_k^\eps)^2}_{\RL^\infty(\del \tilde \Omega^\eps \setminus \del T^\eps)}
\end{equation}
We have as in the proof of Lemma \ref{prop:limitproblem} that the $\RL^2$ norm of $u_k^\eps$ is uniformly bounded, from uniform boundedness of the trace operator. From Lemma \ref{lemma:sobolevpsieps} we have that $$\eps \norm{\del_\nu
\Psi_\eps}_{\RL^\infty(\del\tilde \Omega^\eps \setminus \del T^\eps)} =
\bigo{\eps^{\frac 1 2 + \frac 1 d}},$$ and as earlier, we have
$$ \norm{(u_k^\eps)^2}_{\RL^\infty(\del \tilde \Omega^\eps \setminus \del
T^\eps)} \le C,$$
from Lemma \ref{lem:tracebv}. Finally, it follows from standard lattice packing
theory that $\# \tilde I^\eps \le C \eps^{1-d}$
for some $C$ depending only on $\Omega$, so that $\abs{\del \tilde \Omega^\eps
\setminus \del T^\eps} \le C$. Combining these three estimates, we have indeed
that the product in
\eqref{eq:delnupsiuk2} is going to $0$ as $\eps \to 0$, concluding the proof.
\end{proof}

Until now, we have shown that the harmonic extensions to the holes in
$\Omega^\eps$
of Steklov eigenpairs $(\sigma_k^{\eps},u_k^{\eps})$ converge weakly in $\RH^1$
and strongly in $\RL^2$ to a solution $(\Sigma,U)$ to the
problem
\begin{equation}
  \begin{cases}
   - \Delta U = A_d \beta \Sigma U & \text{in } \Omega, \\
    \del_\nu U = \Sigma U & \text{on } \del \Omega.
  \end{cases}
\end{equation}
It remains to be shown that the convergence is to the ``right'' eigenpair
$(\Sigma_{k,\beta},U_k)$. Spectral convergence of this type is a staple of
  homogenisation theory, see for example \cite{ACG,NPP} in a standard setting or
\cite{FL} using the theory of $E$-convergence. In both cases,
the general theory cannot be directly applied here since the Hilbert space
$\RL^2_{A_d\beta}(\Omega) \times \RL^2(\del \Omega)$ on which the
limit problem is self-adjoint has no natural embedding to the Hilbert spaces
$\RL^2(\del \Omega^\eps)$, compare e.g. with \cite[Definitions 1--3]{FL}. Our methods use instead the quadratic forms
associated with the eigenproblems directly, similar methods were used in
spectral prescription, see e.g. \cite{CdV87}.

For the remainder of this section, as $\beta$ is
fixed, we will write simply $\Sigma_k$.
We first start by the following lemma, showing that the limit function $U$ does
not degenerate to the $0$ function.
\begin{lemma}
  Let $U$ be such that $U_k^{\eps} \to U$ weakly in $\RH^1(\Omega)$. Then, 
  \begin{equation}
    (U,U)_\beta = 1.
  \end{equation}
\end{lemma}
\begin{proof}
 By compactness of the
  trace operator on $\del \Omega$, we have that
  \begin{equation}
    \int_{\del \Omega} (U_k^{\eps})^2 \de \bx \to \int_{\del \Omega} U^2 \de
    \bx.
  \end{equation}
  as $\eps \to 0$. Moreover, by Lemma \ref{lem:normconv}, we have that
  \begin{equation}
    \int_{\del T^\eps} \left( U_k^\eps \right)^2 \de \bx \to A_d\beta\int_\Omega U^2\de\bx 
  \end{equation}
  as $\eps \to 0$. Hence $(U_k^\eps,U_k^\eps)_{\del^\eps} \to (U,U)_\beta$.
  Since $U_k^\eps$ has been normalised to $\RL^2(\del \Omega^\eps)$ norm $1$, this concludes the
  proof.
\end{proof}
We are now ready to complete the proof of our first main result.
\begin{proof}[Proof of Theorem \ref{thm:homogenisation}]
We first show that all the eigenvalues converge. We proceed by induction on the eigenvalue rank $k$. The case $k = 0$ is
trivial. Indeed, we then have that the eigenvalues $\sigma_0^{\eps} \equiv 0$
obviously converge to $\Sigma_0 = 0$ and the normalised eigenfunctions $U_0^{\eps}(\bx) = \abs{\del
\Omega^\eps}^{-1/2}$, which converges to the constant fonction 
$$U_0(\bx) =
\left(\abs{\del \Omega} + A_d\beta\abs{\Omega}\right)^{-1/2}.$$
Suppose now that for all $0 \le j \le k-1$, we have that $U_j^{\eps}$ converges to $U_j$ weakly in $\RH^1(\Omega)$.
We first show that 
\begin{equation}
  \label{eq:asympboundsigstar}
  \Sigma_k \ge \sigma_k^{\eps} + \smallo 1.
\end{equation}
In order to do so, we will show that the eigenfunctions $U_k$ are good
approximations to appropriate test functions for the variational
characterisation \eqref{eq:varcharsigm} of $\sigma_k^\eps$.

Observe that by compactness of the trace operator on $\del \Omega$ and by Lemma
\ref{prop:limitproblem}
\begin{equation}
  \lim_{\eps \to 0} \int_{\del \Omega^\eps} u_j^{\eps} U_k \de A = \int_{\del
  \Omega} U_k U_j\de A + A_d\beta \int_\Omega U_k U_j \de \bx= 0
\end{equation}
for all $0 \le j \le k-1$. Hence, we can write
\begin{equation}
  \label{eq:ukalmostortho}
  U_k = V^{\eps} + \sum_{j=0}^{k-1} \eta_j^{\eps} u_j^{\eps} \de A
\end{equation}
where for all $0 \le j \le k-1$ and all $\eps > 0$, 
\begin{equation}
 \int_{\del \Omega^\eps} V^\eps u_j^\eps = 0
\end{equation}
and $\eta_j^{\eps} \to 0$ as $\eps \to 0$.
 Now, we have that
\begin{equation}
  \begin{aligned}
    \Sigma_k &= \int_\Omega \abs{\nabla U_k}^2 \de \bx \\
    &\ge \int_{\Omega^\eps} \abs{\nabla V^{\eps}}^2  + \sum_{j=0}^{k-1}(\eta_j^{\eps}) \nabla
    u_j^{\eps} \cdot \nabla V^\eps + \sum_{j,l = 0}^{k-1} \eta_j^\eps
    \eta_k^\eps \nabla u_j^\eps \cdot \nabla u_k^\eps \de \bx. \\
\end{aligned}
\end{equation}
Since the $u_j^\eps$ and $V^\eps$ are bounded in $\RH^1(\Omega)$, the integral
of the two sums in the previous equation go to $0$ as $\eps \to 0$. On the other
hand, we have that for all $\eps > 0$, $V^\eps$ is an appropriate test function
for $\sigma_k^\eps$. Hence,
\begin{equation}
  \int_{\Omega^\eps} \abs{\nabla V^\eps}^2 \de \bx \ge \sigma_k^\eps \int_{\del
  \Omega^\eps} (V^\eps)^2 \de A.
\end{equation}
It follows from the decomposition \eqref{eq:ukalmostortho} and
the fact that by Lemma \ref{prop:limitproblem},
\begin{equation}
  \int_{\del \Omega^\eps} U_k^2 \de A \to \int_{\del \Omega} U_k^2 \de A + A_d\beta
  \int_\Omega U_k^2 \de \bx = 1
\end{equation}
that
\begin{equation}
  \lim_{\eps\to 0} \int_{\del \Omega^\eps} (V^\eps)^2 \de \bx  = 1,
\end{equation}
implying that indeed $\Sigma_k \ge \sigma_k^{\eps} + \smallo 1$. We now show that 
$$
\Sigma_k \le \sigma_k^{\eps}.
$$
Let $(\Sigma,U)$ be the limit eigenpair for
$(\sigma_k^{\eps},u_k^{\eps})$ and suppose that 
$$(\Sigma,U) =
(\Sigma_j,U_j)$$
for some $0 \le j \le k-1$. We have
that  
\begin{equation}
  \begin{aligned}
  0 &=   \lim_{\eps \to 0} \int_{\del \Omega^\eps} u_k^{\eps} \de A \\&= 
  \lim_{\eps \to 0}\int_{\del \Omega^\eps}
  u_k^\eps U_j \de A + \int_{\del \Omega^\eps} u_k^\eps( u_j^\eps - U_j) \de A
\end{aligned}
\end{equation}
The first term converges to $1$ by the assumption that $(U,U_j)_\beta = 1$ and Lemma
\ref{prop:limitproblem}. As for the second term, we have by the Cauchy--Schwarz
inequality and the normalisation $\norm{u_k^\eps}_{\RL^2(\del \Omega^\eps)}=1$ that
\begin{equation}
  \begin{aligned}
  \int_{\del \Omega^\eps} u_k^\eps( u_j^\eps - U_j) \de A &\le \norm{u_j^\eps -
  U_j}_{\RL^2(\del \Omega^\eps)} \\
  &\to 0
\end{aligned}
\end{equation}
as $\eps \to 0$ by Lemma \ref{lem:normconv}. This results in a contradiction. We therefore deduce that $\sigma_k^{\eps}$
converges to some eigenvalue $\Sigma$ of problem \ref{problem:homo} that
is larger than $\Sigma_j$ for all $j \le k-1$. Combining this with the
bound \eqref{eq:asympboundsigstar} implies that $\sigma_k^{\eps}$ converges to
$\Sigma_k$, and convergence of the eigenfunction therefore follows, when the
eigenvalues are simple. 

Otherwise, if $\Sigma_k$ has multiplicity $m$, \emph{i.e.} 
\begin{equation}
  \Sigma_{k-1} < \Sigma_k = \dotso = \Sigma_{k+m-1} < \Sigma_{k+m},
\end{equation}
observe that the above argument still yields convergence of $\sigma_j^\eps$ to
$\Sigma_j$ for all $k \le j < k+m$. For the eigenfunctions, start by fixing a
basis $U_k,\dotsc,U_{k+m - 1}$ of the eigenspace associated with $\Sigma_k$. Observe that along
any subsequence there is a further subsequence such that all the eigenfunctions
$U_j^\eps$, converge simultaneously to solutions of Problem \ref{problem:homo}.
Since for all $\eps$, the functions $U_j^\eps$ were $\RL^2(\del
\Omega^\eps)$-orthogonal, in the limit they are still orthogonal, this time with respect
to $(\cdot,\cdot)_\beta$. This implies that in the limit they span the
eigenspace associated with $\Sigma_k$. As such, the projection on the span of
$\set{U_j^\eps : k \le j < k+m}$ converges to the projections on the span of
$\set{U_j : k \le j < k+m}$. Since this was
true along any subsequence, it is also true for the whole sequence, proving
convergence of the projections in the sense alluded to in Remark \ref{rem:multiple}.
\end{proof}
\section{\bf Dynamical boundary conditions with large parameter}\label{Section:LargeBeta}
The goal of this final section is to understand the limit as $\beta$ becomes large of the eigenvalues $\Sigma_{k,\beta}$ and of the corresponding eigenfunctions $U_{k,\beta}$, normalized by
\begin{equation*}
 1 = (U_{k,\beta},U_{k,\beta})_\beta = \int_{\del \Omega} U_{k,\beta}^2 \de A + A_d \beta \int_{\Omega} U_{k,\beta}^2 \de \bx.
\end{equation*}
Recall that the Neumann eigenvalues of $\Omega$ are
$$0=\mu_0\leq\mu_1\leq\mu_2\leq\cdots\nearrow\infty.$$
We are now ready to prove our second main result.

\begin{proof}[Proof of Theorem \ref{thm:limitbeta}]
For $k$ fixed, we start by showing that $\beta \Sigma_{k,\beta}$ is bounded. Consider the min--max characterisation of $\Sigma_{k,\beta}$ to obtain
\begin{equation}
\label{eq:minmaxbSig}
 \tilde \Sigma_{k,\beta}:=\beta \Sigma_{k,\beta} = \min_{\substack{E \subset \RH^1(\Omega) \\ \operatorname{dim}(E) = k+1}} \max_{f \in E \setminus \set 0} \frac{\int_{\Omega} \abs{\nabla f}^2 \de \bx}{\frac 1 \beta \int_{\del \Omega} f^2 \de A + A_d \int_\Omega f^2 \de \bx}.
\end{equation}
The quotient on the righthand side of \eqref{eq:minmaxbSig} is clearly bounded uniformly in $\beta$ for any $k+1$ dimensional subspace of smooth functions on $\Omega$. 
We can therefore suppose that a subsequence in $\beta$ of $\tilde \Sigma_{k,\beta}$ converges, say to $\tilde \Sigma_{k,\infty}$. Let us now prove that $\tilde U_{k,\beta} := \beta^{1/2} U_{k,\beta}$ is a bounded family in $\RH^1(\Omega)$. The normalisation on $U_{k,\beta}$ implies that
\begin{equation*}
 1 \ge A_d \beta\int_\Omega U_{k,\beta}^2 \de \bx 
  = A_d \int_\Omega \tilde U_{k,\beta}^2 \de \bx.
\end{equation*}
For the Dirichlet energy, we have that
\begin{equation}
 \int_{\Omega} \abs{\nabla\tilde U_{k,\beta}}^2 \de \bx = \beta \Sigma_{k,\beta},
\end{equation}
which was already shown to be bounded. Therefore there is also a weakly convergent subsequence in in $\RH^1(\Omega)$ as $\beta \to \infty$, converging to say $\tilde U_{k,\infty}$. We take the subsequence to coincide with the one for $\tilde \Sigma_{k,\beta}$. Observe as well that the normalisation condition on $U_{k,\beta}$ prevents the limit $\tilde U_{k,\infty}$ from vanishing identically.

The functions $\tilde U_{k,\beta}$ satisfy the following weak variational characterisation for any element $V \in \RH^1(\Omega)$
\begin{equation}
\int_{\Omega} \nabla \tilde U_{k,\beta} \cdot \nabla V\,\de\bx = \tilde\Sigma_{k,\beta}\left(\beta^{-1}\int_{\del \Omega} \tilde U_{k,\beta} V \de A + A_d \int_\Omega \tilde U_{k,\beta} V \de \bx  \right).
\end{equation}
Letting $\beta \to \infty$, weak convergence of $\tilde U_{k,\beta}$ in $\RH^1(\Omega)$ implies that the limit satisfies the weak identity
\begin{equation}
  \int_\Omega \nabla \tilde U_{k,\infty} \cdot\nabla V\,\de\bx = \tilde \Sigma_{k,\infty} A_d \int_\Omega \tilde U_{k,\infty} V \de \bx. 
\end{equation}
In other words, $\tilde U_{k,\infty}$ is a solution to the Neumann eigenvalue problem with eigenvalue $\mu = \tilde \Sigma_{k,\infty} A_d$.

We now proceed by recursion on $k$ to show convergence to the right eigenpair. Once again, the statement is trivial for $k = 0$ and the constant eigenfunction. Assume that we have convergence for the first $k-1$ eigenpairs.
We now proceed in a similar fashion as in the proof of the spectral convergence to Problem \eqref{problem:homo}. We repeat the argument because the inequalities are more subtle.
 We first show that
 \begin{equation}
 \label{eq:neumannupper}
  \tilde \mu_k(\Omega):= \frac{\mu_k(\Omega)}{A_d} \ge \tilde \Sigma_{k,\beta}(1 + \smallo 1).
 \end{equation}

 Write
 \begin{equation}
  f_k = F_\beta + \sum_{j=0}^{k-1} (f_k, U_{j,\beta})_\beta U_{j,\beta},
 \end{equation}

with $F_\beta \perp_\beta U_{j,\beta}$ for $0 \le j < k$. We have that
\begin{equation}
\label{eq:scalardecomp}
 (f_k,U_{j,\beta})_\beta = (f_k,\beta^{-1/2}f_j)_\beta + (f_k,U_{j,\beta}-\beta^{-1/2}f_j)_\beta .
\end{equation}
The first inner product develops as
\begin{equation}
 (f_k,\beta^{-1/2}f_j)_\beta = \beta^{-1/2} \int_{\del\Omega} f_k f_j \de A +A_d\beta^{1/2}\int_{\Omega} f_k f_j \de \bx. 
\end{equation}
The first term clearly goes to $0$ as $\beta \to \infty$, and the second one vanishes by orthogonality of the Neumann eigenfunctions in $\RL^2(\Omega)$. We now turn our attention to the second inner product in \eqref{eq:scalardecomp}. We have that
\begin{equation}
 \lim_{\beta \to \infty} \beta^{-1/2} \int_{\del \Omega}f_k (\tilde U_{j,\beta} - f_j) \de A = 0
\end{equation}
by weak $\RH^1$ convergence of $\tilde U_{j,\beta}$ and compactness of the trace operator. On the other hand, strong $\RL^2$ convergence implies that
\begin{equation}
\label{eq:tightcontrol}
 A_d \beta^{1/2} \int_{\Omega} f_k (\tilde U_{j,\beta} - f_j) \de \bx = \smallo{\beta^{1/2}}.
\end{equation}
All in all, this implies that
\begin{equation}
 (f_k,U_{j,\beta})_\beta = \smallo{\beta^{1/2}}
\end{equation}
for all $0 \le j < k$. 
We now write
\begin{equation}
\label{eq:firstapproxneumann}
 \begin{aligned}
  \tilde \mu_k &= \frac{1}{A_d} \int_\Omega \abs{\nabla f_k}^2 \de \bx \\
  &\ge \frac{1}{A_d} \int_{\Omega} \abs{\nabla F_\beta}^2 \de \bx - \frac{1}{A_d}\sum_{j=0}^{k-1} (f_k,U_{j,\beta})_\beta^2 \Sigma_{j,\beta}.
 \end{aligned}
\end{equation}
Since $\beta \Sigma_{j,\beta}$ is bounded and equation \eqref{eq:tightcontrol} implies that $(f_k,U_{j,\beta})_\beta^2 = \smallo{\beta}$, we deduce that the last term in \eqref{eq:firstapproxneumann} goes to $0$. We now observe that by the variational characterisation of $\Sigma_{k,\beta}$,
\begin{equation}
\begin{aligned}
  \frac{1}{A_d}\int_{\Omega} \abs{\nabla F_\beta}^2 \de \bx &\ge \frac{\Sigma_{k,\beta}}{A_d}\left(\int_{\del \Omega} F_\beta^2 + A_d \beta \int_\Omega F_\beta^2\right), \\
 &= \tilde \Sigma_{k,\beta} \left(\frac{1}{A_d\beta}\int_{\del\Omega} F_\beta^2 \de A  + \int_\Omega F_\beta^2 \de \bx\right).
 \end{aligned}
\end{equation}
In the same way as we obtained the bound on the last term in \eqref{eq:firstapproxneumann}, the first integral is $\smallo \beta$. As for the second one, we write
\begin{equation*}
 \begin{aligned}
  \int_\Omega F_\beta^2  \de \bx &= \int_{\Omega} \left(f_k - \sum_{j=0}^{k-1} (f_k, U_{j,\beta})_\beta U_{j,\beta}^2\right)^2 \de \bx \\
  &= 1 - \sum_{j=0}^{k-1} \int_\Omega (f_k,U_{j,\beta})_\beta f_j U_{j,\beta}\de \bx + \sum_{j,\ell = 0}^{k-1} \int_\Omega (f_k,U_{j,\beta})_\beta(f_k,U_{\ell,\beta})_\beta U_{j,\beta}U_{\ell,\beta} \de \bx.
 \end{aligned}
\end{equation*}
Strong $\RL^2(\Omega)$ convergence of $\beta^{1/2} U_{j,\beta}$ and the fact that $(f_k,U_{j,\beta}) = \smallo{\beta^{1/2}}$ imply that the last two integrals converge to $0$ as $\beta \to \infty$. 

We have therefore obtained that
\begin{equation}
 \tilde \mu_k(\Omega) \ge \tilde \Sigma_{k,\beta}\left(1 + \smallo 1\right),
\end{equation}
\emph{i.e.} we have indeed proven assertion \eqref{eq:neumannupper}.

Suppose now that $(\tilde \Sigma_{k,\beta}, \tilde U_{k,\beta})$ converge to a Neumann eigenpair $(\tilde \mu_j, f_j)$ for some $j < k$. Then, we have that 
\begin{equation}
 \begin{aligned}
  1 &= \lim_{\beta \to \infty} \int_\Omega f_j \tilde U_{k,\beta} \de \bx \\
  &= \lim_{\beta \to \infty} \int_{\Omega} (f_j - \tilde U_{j,\beta})\tilde U_{k,\beta} \de x + \int_\Omega \tilde U_{k,\beta} U_{j,\beta} \de \bx \\
  &\le \lim_{\beta \to \infty}\norm{f_j - \tilde U_{j,\beta}}_{\RL^2(\Omega}\norm{\tilde U_{k,\beta}}_{\RL^2(\Omega)} + \abs{\int_\Omega U_{k,\beta} U_{j,\beta} \de \bx} \\
  &= 0,
 \end{aligned}
\end{equation}
where the limit comes from strong convergence in $\RL^2(\del \Omega)$ to $0$ of $U_{j,\beta}$ and our recursion hypothesis. This is a contradiction, hence the convergence is to the correct eigenpair. This implies convergence of the whole sequence if the Neumann spectrum of $\Omega$ is simple. If there are eigenvalues with multiplicity, the same procedure as for the homogenisation problem yields once again convergence.

As for continuity in $\beta$, the same proof goes through in exactly the same way, except for the fact that we do not need to show the boundedness results in $\beta$.
\end{proof}

We can now prove the comparison results between Steklov and Neumann eigenvalues.

\begin{proof}[Proof of Corollary \ref{cor:boundd}]

It is proved in~\cite[Theorem 1.4]{ceg2} that any bounded domain $\Omega\subset\R^d$ satisfies
\begin{gather}
  \sigma_k(\Omega)|\partial\Omega|\leq C(d)|\Omega|^{\frac{d-2}{d}}k^{2/d},
\end{gather}
where $C(d)$ is a constant which depends only on the dimension.
When applied to $\Omega^\eps$ this leads, after taking $\eps\to 0$, to
\begin{gather}
  \Sigma_{k,\beta}(|\partial\Omega|+A_d\beta|\Omega|)\leq C(d)|\Omega|^{\frac{d-2}{d}}k^{2/d}.
\end{gather}
Taking the limit $\beta\to\infty$ leads to
\begin{gather}
  \mu_k|\Omega|\leq C(d)|\Omega|^{\frac{d-2}{d}}k^{2/d}.
\end{gather}
This is equivalent to
\begin{gather}
  \mu_k|\Omega|^{2/d}\leq C(d)k^{2/d}.
\end{gather}
\end{proof}

Finally, all is left to do is to prove 
the previous theorem has the following corollary in dimension $d = 2$. It allows one to transform universal bounds for Steklov eigenvalues into universal bounds for Neumann eigenvalues.

We write $\Omega^\eps_\beta$ for a domain $\Omega^\eps$ as constructed earlier whose holes are exactly of radius $r_\eps^{d-1} = \beta \eps^{d}$.

\begin{proof}[Proof of Theorem \ref{thm:neumannd=2}]
 We have from Theorem \ref{thm:homogenisation} that
 \begin{equation}
 \begin{aligned}
  \lim_{\eps \to 0} \sigma_k\left(\Omega_\beta^\eps\right)\abs{\del \Omega_\beta^\eps} &= \Sigma_{k,\beta}(\Omega)\left(\abs{\del \Omega} + A_d \beta \abs \Omega\right) \\
  &= \frac{\abs{\del \Omega}} \beta \tilde \Sigma_{k,\beta} + A_d \abs \Omega \tilde \Sigma_{k,\beta}. 
  \end{aligned}
 \end{equation}
Now, the first term clearly goes to $0$ as $\beta \to \infty$ while, by Theorem \ref{thm:limitbeta}, we have that
\begin{equation}
 \lim_{\beta \to \infty} A_d \abs \Omega \tilde \Sigma_{k,\beta} = \mu_k(\Omega) \abs \Omega.
\end{equation}
\end{proof}

\subsection*{Acknowledgements}
The authors would like to thank Iosif Polterovich for several very useful
conversations. We would also like to thank Valeri Smyshlyaev and Ilya Kamotski
for pointing us towards many classical results in homogenisation theory. We
would also like to thank Almut Burchard for possible further related questions.
The authors also would like to thank Gérard Philippin, who was helpful in the
early stage of this project. 
We thank both referees for a careful reading and
many comments which greatly helped in improving the paper. In particular, we
thank one of them for comments that led to Lemma \ref{lem:tracebv}, and the
other for pointing out reference \cite{HassannezhadMiclo}.
A.G. acknowledges support from NSERC. A.H. was partially supported by the ANR
project ANR-18-CE40-0013 -- SHAPO on Shape Optimisation. He wants
also to thank the Centre de Recherche Math\'ematiques de Montr\'eal and the
Département de Mathématiques of Université Laval for grants and fruitful
atmosphere during his stay there in 2018. 
J.L. was supported by the NSERC postdoctoral fellowship, and EPSRC grant number EP/P024793/1.

\bibliographystyle{plain}
\bibliography{../homo}

\def\cprime{$'$} \def\cprime{$'$} \def\cprime{$'$} \def\cprime{$'$}
  \def\cprime{$'$} \def\cprime{$'$}
\begin{thebibliography}{10}

\bibitem{allaire}
G.~Allaire.
\newblock {\em Shape optimization by the homogenization method}, volume 146 of
  {\em Applied Mathematical Sciences}.
\newblock Springer-Verlag, New York, 2002.

\bibitem{ACG}
Y.~Amirat, G.~A. Chechkin, and R.~R. Gadyl{\cprime}shin.
\newblock Asymptotics of simple eigenvalues and eigenfunctions for the
  {L}aplace operator in a domain with oscillating boundary.
\newblock {\em Zh. Vychisl. Mat. Mat. Fiz.}, 46(1):102--115, 2006.

\bibitem{Arrieta2008}
J.~Arrieta, \'{A}. Jim\'{e}nez-Casas, and A.~Rodr\'{\i}guez-Bernal.
\newblock Flux terms and {R}obin boundary conditions as limit of reactions and
  potentials concentrating at the boundary.
\newblock {\em Rev. Mat. Iberoam.}, 24(1):183--211, 2008.

\bibitem{bucurfreitaskennedy}
D.~Bucur, P.~Freitas, and J.~Kennedy.
\newblock The {R}obin problem.
\newblock In {\em Shape optimization and spectral theory}, pages 78--119. De
  Gruyter Open, Warsaw, 2017.

\bibitem{bgt}
D.~Bucur, A.~Giacomini, and P.~Trebeschi.
\newblock ${L}^\infty$ bounds of {S}teklov eigenfunctions and spectrum
  stability under domain variations.
\newblock cvgmt preprint.

\bibitem{chaveleig}
I.~Chavel.
\newblock {\em Eigenvalues in {R}iemannian geometry}, volume 115 of {\em Pure
  and Applied Mathematics}.
\newblock Academic Press, Inc., Orlando, FL, 1984.
\newblock Including a chapter by Burton Randol, With an appendix by Jozef
  Dodziuk.

\bibitem{CheredDondlRosler}
K.~Cherednichenko, P.~Dondl, and F.~R\"{o}sler.
\newblock Norm-resolvent convergence in perforated domains.
\newblock {\em Asymptot. Anal.}, 110(3-4):163--184, 2018.

\bibitem{ChiadoPiatNazarovPiatnitski}
V.~Chiado~Piat, S.~S. Nazarov, and A.~L. Piatnitski.
\newblock Steklov problems in perforated domains with a coefficient of
  indefinite sign.
\newblock {\em Netw. Heterog. Media}, 7(1):151--178, 2012.

\bibitem{CianciGirouard}
D.~Cianci and A.~Girouard.
\newblock Large spectral gaps for {S}teklov eigenvalues under volume
  constraints and under localized conformal deformations.
\newblock {\em Ann. Global Anal. Geom.}, 54(4):529--539, 2018.

\bibitem{CioranescuMurat}
D.~Cioranescu and F.~Murat.
\newblock Un terme \'{e}trange venu d'ailleurs.
\newblock In {\em Nonlinear partial differential equations and their
  applications. {C}oll\`ege de {F}rance {S}eminar, {V}ol. {II} ({P}aris,
  1979/1980)}, volume~60 of {\em Res. Notes in Math.}, pages 98--138, 389--390.
  Pitman, Boston, Mass.-London, 1982.

\bibitem{ceg2}
B.~Colbois, A.~El~Soufi, and A.~Girouard.
\newblock Isoperimetric control of the {S}teklov spectrum.
\newblock {\em J. Funct. Anal.}, 261(5):1384--1399, 2011.

\bibitem{CEG3}
B.~Colbois, A.~El~Soufi, and A.~Girouard.
\newblock Compact manifolds with fixed boundary and large {S}teklov
  eigenvalues.
\newblock {\em Proc. Amer. Math. Soc.}, 147(9):3813--3827, 2019.

\bibitem{CdV87}
Y.~Colin~de Verdi{\`e}re.
\newblock Construction de laplaciens dont une partie finie du spectre est
  donn\'ee.
\newblock {\em Ann. Sci. \'Ecole Norm. Sup. (4)}, 20(4):599--615, 1987.

\bibitem{Douanla}
H.~Douanla.
\newblock Homogenization of {S}teklov spectral problems with indefinite density
  function in perforated domains.
\newblock {\em Acta Appl. Math.}, 123:261--284, 2013.

\bibitem{FL}
A.~Ferrero and P.~D. Lamberti.
\newblock Spectral stability for a class of fourth order {S}teklov problems
  under domain perturbations.
\newblock {\em Calc. Var. Partial Differential Equations}, 58(1):Art. 33, 57,
  2019.

\bibitem{fraschoen3}
A.~Fraser and R.~Schoen.
\newblock Minimal surfaces and eigenvalue problems.
\newblock In {\em Geometric analysis, mathematical relativity, and nonlinear
  partial differential equations}, volume 599 of {\em Contemp. Math.}, pages
  105--121. Amer. Math. Soc., Providence, RI, 2013.

\bibitem{fraschoen2}
A.~Fraser and R.~Schoen.
\newblock Sharp eigenvalue bounds and minimal surfaces in the ball.
\newblock {\em Invent. Math.}, 203(3):823--890, 2016.

\bibitem{FreiLauW2018}
P.~Freitas and R.~S. Laugesen.
\newblock From {N}eumann to {S}teklov and beyond, via {R}obin: the {W}einberger
  way.
\newblock {\em Amer. J. Math}, to appear.

\bibitem{FreiLauS2018}
P.~Freitas and R.~S. Laugesen.
\newblock From {S}teklov to {N}eumann and beyond, via {R}obin: the {S}zeg{\H o}
  way.
\newblock {\em Canad. J. Math.}, to appear.

\bibitem{gilbargtrudinger}
D.~Gilbarg and N.~S. Trudinger.
\newblock {\em Elliptic partial differential equations of second order}.
\newblock Classics in Mathematics. Springer-Verlag, Berlin, 2001.
\newblock Reprint of the 1998 edition.

\bibitem{gpsurvey}
A.~Girouard and I.~Polterovich.
\newblock Spectral geometry of the {S}teklov problem (survey article).
\newblock {\em J. Spectr. Theory}, 7(2):321--359, 2017.

\bibitem{GryCristo2014}
S.~Gryshchuk and M.~Lanza~de Cristoforis.
\newblock Simple eigenvalues for the {S}teklov problem in a domain with a small
  hole. {A} functional analytic approach.
\newblock {\em Math. Methods Appl. Sci.}, 37(12):1755--1771, 2014.

\bibitem{HassannezhadMiclo}
A.~Hassannezhad and L.~Miclo.
\newblock Higher order {C}heeger inequalities for {S}teklov eigenvalues.
\newblock {\em Ann. Sci. \'Ecole Norm. Sup. (4)}, to appear.

\bibitem{JikovKozlovOlenik}
V.~V. Jikov, S.~M. Kozlov, and O.~A. Oleinik.
\newblock {\em Homogenization of differential operators and integral
  functionals}.
\newblock Springer-Verlag, Berlin, 1994.
\newblock Translated from the Russian by G. A. Yosifian [G. A. Iosif\cprime
  yan].

\bibitem{Kaizu}
S.~Kaizu.
\newblock The {R}obin problems on domains with many tiny holes.
\newblock {\em Proc. Japan Acad. Ser. A Math. Sci.}, 61(2):39--42, 1985.

\bibitem{Kroger1992}
P.~Kr\"{o}ger.
\newblock Upper bounds for the {N}eumann eigenvalues on a bounded domain in
  {E}uclidean space.
\newblock {\em J. Funct. Anal.}, 106(2):353--357, 1992.

\bibitem{LambertiProvenzano2015}
P.~D. Lamberti and L.~Provenzano.
\newblock Viewing the {S}teklov eigenvalues of the {L}aplace operator as
  critical {N}eumann eigenvalues.
\newblock In {\em Current trends in analysis and its applications}, Trends
  Math., pages 171--178. Birkh\"{a}user/Springer, Cham, 2015.

\bibitem{LambertiProvenzano2017}
P.~D. Lamberti and L.~Provenzano.
\newblock Neumann to {S}teklov eigenvalues: asymptotic and monotonicity
  results.
\newblock {\em Proc. Roy. Soc. Edinburgh Sect. A}, 147(2):429--447, 2017.

\bibitem{MarchenkoKhruslov1964}
V.~A. Mar\v{c}enko and E.~Ya. Khruslov.
\newblock Boundary-value problems with fine-grained boundary.
\newblock {\em Mat. Sb. (N.S.)}, 65 (107):458--472, 1964.

\bibitem{MarchenkoKhruslov}
V.~A. {Mar\v{c}enko} and E.~Ya. {Khruslov}.
\newblock {\em {{Boundary value problems in domains with a fine-grained
  boundary.}}}
\newblock Izdat. Naukova Dumka, Kiev, 1974.

\bibitem{dmfz}
G.~Dal Maso, G.~Franzina, and D.~Zucco.
\newblock Transmission conditions obtained by homogenisation.
\newblock {\em Nonlinear Analysis}, 177:361 -- 386, 2018.
\newblock Nonlinear PDEs and Geometric Function Theory, in honor of Carlo
  Sbordone on his 70th birthday.

\bibitem{Nazarov2014}
S.~A. Nazarov.
\newblock Asymptotic expansions of eigenvalues of the {S}teklov problem in
  singularly perturbed domains.
\newblock {\em Algebra i Analiz}, 26(2):119--184, 2014.

\bibitem{NPP}
S.~A. Nazarov, I.~L. Pankratova, and A.~L. Piatnitski.
\newblock Homogenization of the spectral problem for periodic elliptic
  operators with sign-changing density function.
\newblock {\em Arch. Ration. Mech. Anal.}, 200(3):747--788, 2011.

\bibitem{RauchTaylor}
J.~Rauch and M.~E. Taylor.
\newblock Potential and scattering theory on wildly perturbed domains.
\newblock {\em J. Funct. Anal.}, 18:27--59, 1975.

\bibitem{shubin}
M.~A. Shubin.
\newblock {\em Pseudodifferential operators and spectral theory}.
\newblock Springer-Verlag, Berlin, second edition, 2001.
\newblock Translated from the 1978 Russian original by Stig I. Andersson.

\bibitem{TaylorI}
M.~E. Taylor.
\newblock {\em Partial differential equations {I}. {B}asic theory}, volume 115
  of {\em Applied Mathematical Sciences}.
\newblock Springer, New York, second edition, 2011.

\bibitem{vanninathan}
M.~Vanninathan.
\newblock Homogenization of eigenvalue problems in perforated domains.
\newblock {\em Proc. Indian Acad. Sci. Math. Sci.}, 90(3):239--271, 1981.

\bibitem{BelowFrancois}
J.~von Below and G.~Fran\c{c}ois.
\newblock Spectral asymptotics for the {L}aplacian under an eigenvalue
  dependent boundary condition.
\newblock {\em Bull. Belg. Math. Soc. Simon Stevin}, 12(4):505--519, 2005.

\end{thebibliography}

\end{document}